\documentclass[reqno]{amsart}

\usepackage{mathtools,amsmath, amsfonts, amssymb, latexsym, amsthm,
  amscd, cancel, enumerate, rotating, comment, appendix, ulem,
  makecell, paralist,times,graphicx,pgf,tikz,tikz-cd,geometry}
\usepackage{tkz-euclide}
\usepackage[capitalise, noabbrev]{cleveref}
\usepackage[latin1]{inputenc}
\usepackage[english]{babel}
\usepackage[mathscr]{eucal}
\usepackage{xxcolor}
\usepackage{caption}
\usepackage{subcaption}

\usepackage[T1]{fontenc}      

\input{xy}
\xyoption{all}

\usepackage[all]{xy}
\usetikzlibrary{shapes}
\usetikzlibrary{snakes}
\usetikzlibrary{arrows}
\usetikzlibrary{matrix, decorations.pathmorphing}
\usetikzlibrary{decorations.markings}

\usetikzlibrary{calc}

\newif\ifdviwin

\newtheorem{theorem}{Theorem}[section]
\newtheorem{lemma}[theorem]{Lemma}
\newtheorem{corollary}[theorem]{Corollary} 
\newtheorem{proposition}[theorem]{Proposition}
\newtheorem*{question}{Question}

\theoremstyle{definition}
\newtheorem{definition}[theorem]{{Definition}}
\newtheorem{example}[theorem]{Example}
\newtheorem{remark}[theorem]{Remark}

\newtheorem{notation}[theorem]{Notation}

\theoremstyle{remark}
\newtheorem*{claim}{Claim}

\newtheorem*{chunk*}{}

\numberwithin{equation}{theorem}

\newcommand{\N}{\mathcal{N}}
\newcommand{\Nrq}{\mathcal{N}^r_q}
\newcommand{\st}{\colon}

\newcommand{\pd}{\mathrm{pd}}
\newcommand{\Taylor}{\mathrm{Taylor}}

\newcommand{\m}{\mathbf{m}}

\newcommand{\bd}{\mathbf{d}}
\newcommand{\f}{\mathbf{f}}
\newcommand{\lcm}{\mathrm{lcm}}
\newcommand{\LCM}{\mathrm{LCM}}
\newcommand{\D}{\Delta}
\newcommand{\LL}{\mathbb{L}}
\newcommand{\Lrq}{\mathbb{L}^r_q}
\newcommand{\Lr}{\mathbb{L}^r}
\newcommand{\tuple}[1]{\langle #1 \rangle}
\newcommand{\Tor}{\mathrm{Tor}}

\newcommand{\E}{\mathcal{E}}
\newcommand{\Erq}{{\E_q}^r}

\newcommand{\NN}{\mathbb{N}}
\newcommand{\ZZ}{\mathbb{Z}}
\newcommand{\ba}{\mathbf{a}}
\newcommand{\bb}{\mathbf{b}}
\newcommand{\bc}{\mathbf{c}}

\newcommand{\be}{\mathbf{e}}

\newcommand{\bm}{\mathbf{m}}
\newcommand{\bma}{\mathbf{m}^\ba}
\newcommand{\bmb}{\mathbf{m}^\bb}

\newcommand{\qwhere}{\quad \mbox{ where } \quad }
\newcommand{\qand}{\quad \mbox{ and } \quad }

\newcommand{\qbut}{\quad \mbox{ but } \quad }
\newcommand{\qforsome}{\quad \mbox{ for some } \quad }
\newcommand{\qforall}{\quad \mbox{ for all } \quad }
\newcommand{\qforeach}{\quad \mbox{ for each } \quad }

\newcommand{\e}{\epsilon}
\newcommand{\pme}{{\pmb{\e}}}
\newcommand{\pmea}{\pme^{\ba}}
\newcommand{\pmeb}{\pme^{\bb}}
\newcommand{\pmec}{\pme^{\bc}}
\newcommand{\sfk}{\mathsf k}

\begin{document}
\bibliographystyle{amsplain}

\author[S.~M.~Cooper]{Susan M. Cooper}
\address{Department of Mathematics\\
University of Manitoba\\
520 Machray Hall\\
186 Dysart Road\\
Winnipeg, MB\\
Canada R3T 2N2}
\email{susan.cooper@umanitoba.ca}

\author[S.~El Khoury]{Sabine El Khoury}
\address{Department of Mathematics,
American University of Beirut,
Bliss Hall 315, P.O. Box 11-0236,  Beirut 1107-2020,
Lebanon}
\email{se24@aub.edu.lb}

\author[S.~Faridi]{Sara Faridi}
\address{Department of Mathematics \& Statistics\\
Dalhousie University\\
6316 Coburg Rd.\\
PO BOX 15000\\
Halifax, NS\\
Canada B3H 4R2}
\email{faridi@dal.ca}

\author[S~Mayes-Tang]{Sarah Mayes-Tang}
\address{Department of Mathematics\\
University of Toronto\\
40 St. George Street, Room 6290\\
Toronto, ON \\
Canada M5S 2E4}
\email{smt@math.toronto.edu}

\author[S.~Morey]{Susan Morey}
\address{Department of Mathematics\\
Texas State University\\
601 University Dr.\\
San Marcos, TX 78666\\E.S.A.}
\email{morey@txstate.edu}

\author[L.~M.~\c{S}ega]{Liana M.~\c{S}ega}
\address{Division of Computing, Analytics and Mathematics, 
University of Missouri-Kansas City, Kansas City, MO 64110 U.S.A.}
\email{segal@umkc.edu}

\author[S.~Spiroff]{Sandra Spiroff }
\address{Department of Mathematics,
University of Mississippi,
Hume Hall 335, P.O. Box 1848, University, MS 38677
USA}
\email{spiroff@olemiss.edu}



\title{Simplicial resolutions of powers of square-free monomial ideals}

\begin{abstract}
 The Taylor resolution is almost never minimal for powers of monomial
 ideals, even in the square-free case.  In this paper we introduce a
 smaller resolution for each power of any square-free monomial ideal,
 which depends only on the number of generators of the ideal. More
 precisely, for every pair of fixed integers $r$ and $q$, we construct
 a simplicial complex that supports a free resolution of the $r^{th}$
 power of {\it any} square-free monomial ideal with $q$ generators. The
 resulting resolution is significantly smaller than the Taylor
 resolution, and is minimal for special cases.  Considering the
 relations on the generators of a fixed ideal allows us to further
 shrink these resolutions. We also introduce a
 class of ideals called ``extremal ideals,'' and show
 that the Betti numbers of powers of all square-free monomial ideals
 are bounded by Betti numbers of powers of extremal ideals.  Our
 results lead to upper bounds on Betti numbers of powers of any
 square-free monomial ideal that greatly improve the binomial bounds
 offered by the Taylor resolution.

\end{abstract} 

\maketitle

%%%%%%%%%%%%%%%%%%%%%%%%%%%%%%%%%%%%%%%%%%%%%%%%%%%%%%%%%%%%%%%
\section{Introduction}
\label{sec:introduction}
%%%%%%%%%%%%%%%%%%%%%%%%%%%%%%%%%%%%%%%%%%%%%%%%%%%%%%%%%%%%%%%%%%%
Important insight about the underlying structure of an ideal in a
polynomial ring is gained from a careful analysis of its minimal free
resolution.  As such, significant effort has gone into the development
of methods to compute resolutions.  The approach of leveraging
connections between commutative algebra and other fields, such as
combinatorics and topology, has proven to be quite fruitful.  Diana
Taylor's thesis~\cite{T} initiated the exploration of these
connections, followed by simplicial resolutions (Bayer, Peeva, and
Sturmfels \cite{BPS}), polytopal complexes (Nagel and Reiner
\cite{NR}), and cellular complexes (Bayer and Sturmfels \cite{BS}), to
name just a few. See \cite{MS,OW} for an overview of these developments.

The Taylor resolution is powerful: given \textit{any} ideal $I$
minimally generated by $q$ monomials, Taylor constructed a simplicial
complex $\Taylor(I)$ by labeling the vertices of a $(q-1)$-simplex with
the monomial generators of $I$. She showed that this complex {\bf
  supports} a free resolution of $I$, in the sense that its simplicial
chain complex can be transformed, via a process called {\it
  homogenization}, to a free resolution of $I$, called the {\bf Taylor
  resolution} of $I$.
Unfortunately, even though every monomial ideal has a Taylor
resolution, the Taylor resolution is often far from minimal.  In
particular, for powers of ideals it is almost never minimal due to certain syzygies that are automatically created when taking powers.

The central theme of this paper is to find an analogue for the Taylor
complex for powers of square-free monomial ideals.
We seek a construction that takes the automatically generated non-minimal syzygies 
into account and removes them from the Taylor resolution to produce a
much smaller free resolution of $I^r$ that works for any monomial ideal $I$. Our
ultimate goal is to find a uniform combinatorial structure that
depends only on the number of generators and the power of the ideal.
More precisely, the question at the heart of this paper is the
following:

\begin{question} Given positive integers $r$ and $q$, is it possible to
  find a simplicial complex (considerably) smaller than the simplex
  $\Taylor(I^r)$ that supports a free resolution of $I^r$, where $I$
  is any ideal generated by $q$ monomials in a polynomial ring?
\end{question}

When $r=1$, $\Taylor(I)$ is in fact the optimal answer to the question
above, as there are ideals $I$ for which $\Taylor(I)$ supports a
minimal resolution.  But when $r=2$, the resolution supported on
$\Taylor(I^2)$ is never minimal for any non-principal square-free
monomial ideal $I$~(\cite{L2}).  As expected, the resolution supported on $\Taylor(I^r)$
becomes further from minimal as $r$ grows.

Although Taylor's complex for the $r^{th}$ power of a monomial ideal
with $q$ generators can be quite large - a simplex of dimension
${q+r-1}\choose{r}$ - we can improve the situation considerably by
studying the general relations among the generators of $I^r$ that must
always exist for all monomial ideals regardless of the generating set
for $I$.  Similar to the idea used in Lyubeznik's
resolutions~\cite{L}, in the case of square-free monomial ideals, this
investigation involves detecting and trimming redundant faces of the
Taylor complex $\Taylor(I^r)$, bringing us closer to a minimal
resolution.

To illustrate our underlying process, let $r=2, q=3$ and consider {\it
  any} ideal $I = (m_1, m_2, m_3)$ in the polynomial ring $\sfk[x_1,
  \ldots, x_n]$ where $m_1, m_2$ and $m_3$ are minimal, square-free
monomial generators and $\sfk$ is a field.  Now $I^2 = ({m_1}^2,
{m_2}^2, {m_3}^2, m_1m_2, m_1m_3, m_2m_3)$ and $\Taylor(I^2)$ is a
$5$-dimensional simplex with $6$ vertices, where each vertex is
labeled by a generator of $I^2$ and each face is labeled with the
least common multiple of its vertices. A non-minimal syzygy occurs
when a face and a subface have the same label (see \cref{t:BPS}). When
considering $I^2$, no matter what the monomial generators of $I$ are,
when $i, j, k$ are distinct we always have the following: $$
\lcm({m_i}^2, {m_j}^2)=\lcm({{m_i}^2,m_im_j,{m_j}^2}) \qand
\lcm({m_i}^2, m_jm_k)=\lcm({m_i}^2, m_jm_k,m_im_k).$$ These equalities
lead to non-minimal syzygies in the Taylor resolution of $I^2$, and as
a result the edges
$$\{{m_1}^2, {m_2}^2\}, \{{m_1}^2, {m_3}^2\}, \{{m_2}^2, {m_3}^2\},
\{{m_1}^2, m_2m_3\}, \{{m_2}^2, m_1m_3\}, \{{m_3}^2, m_1m_2\},$$ and
all faces containing these edges, can be removed from
$\Taylor(I^2)$. The resulting $2$-dimensional subcomplex $\LL_3^2$ of
$\Taylor(I^2)$ supports a resolution of $I^2$~(\cite{L2}). This
simple observation has a considerable impact on bounding the Betti numbers of
$I^2$. For example, in this small case we can conclude that for every
ideal $I$ with three square-free monomial generators, the projective
dimension of $I^2$ is at most $2$ (the dimension of $\LL_3^2$), versus
the bound $5$ (the dimension of $\Taylor(I^2)$).

In this paper we show that a similar argument can be made more
generally.  If $I$ is minimally generated by square-free monomials
$m_1,\ldots,m_q$, then we can identify faces of the simplex
$\Taylor(I^r)$ which lead to non-minimal syzygies for $I^r$.
Eliminating these faces results in a much smaller subcomplex of
$\Taylor(I^r)$, which we call $\Lrq$ (\cref{d:Lrq}).

Our main result, \cref{t:main}, shows that $\Lrq$ supports a
resolution of $I^r$.  This echoes the power of Taylor's complex in
that we have a topological structure depending solely on $r$ and $q$
that supports a resolution of $I^r$ for \textit{any} square-free
monomial ideal $I$ with $q$ minimal generators.
By further deleting redundancies specific to the generators of $I$, we obtain a subcomplex $\Lr(I)$ of $\Lrq$, so that we have
$$\Lr(I) \subseteq \Lrq \subseteq \Taylor(I^r).$$ \cref{t:main} also shows
that $\Lr(I)$ supports a free resolution of $I^r$.  Our approach
to proving this involves showing that $\Lr(I)$ and $\Lrq$ are both
\textit{quasi-trees} (see \cref{d:defs}), meaning (roughly) that are
built from a special ordering on their facets. In
\cref{t:res-by-quasitree}, we show that for a quasi-tree to support a
free resolution one only needs to check that certain induced subcomplexes are
connected. This extends the main result of~\cite{F14}.

Our complex $\Lr(I)$ gives natural and useful bounds on homological
information of $I^r$.  Indeed, the Betti numbers of $I^r$ are bounded
in terms of the number and size of faces of $\Lr(I)$, yielding bounds that are significantly
smaller than those given by the Taylor resolution.  
In the later sections of this paper we examine just how much smaller these bounds on the Betti numbers are and when the resolutions obtained are minimal. 
We define a class of square-free monomial ideals, which we call \textit{extremal ideals}, whose Betti numbers of powers bound the Betti numbers of powers of any square-free monomial ideal with the same number of generators. As a result, we reduce the question of minimality of $\Lrq$ to the study of when  $\Lrq$ supports a minimal free resolution of the $r^{th}$ power of an extremal ideal. 

To put this work in context of the broader literature, studying powers of ideals and bounding their invariants has received much attention in recent years. Powers play an important role in Rees algebras and associated graded rings among other uses, making understanding their behavior desirable but difficult. In another direction, there has been considerable interest in describing minimal topological resolutions for all monomial ideals  using a variety of methods, such as using chain maps from multiple simplicial complexes (see \cite{EMO, Tch}). In this paper we combine the two interests and seek, for powers of monomial ideals,  resolutions that are supported on a single topological structure which is practical to determine based on the generators of the original ideal $I$.

The paper is organized as follows. \cref{sec:background} contains
basics of simplicial resolutions. In \cref{sec:res-by-qtrees} we use
simplicial collapsing and the Bayer--Peeva--Sturmfels criterion to
prove the above-mentioned criterion for quasi-trees
(\cref{t:res-by-quasitree}).  In \cref{sec:Lrq} we introduce the
definition of the simplicial complex $\Lrq$, and we prove in
\cref{p:Lr-quasitree} that $\Lrq$ is a quasi-tree.  In \cref{sec:LrI}
we define $\Lr(I)$, discuss some examples, and then prove the main
result, \cref{t:main}. \cref{s:bounds} investigates the bounds on the
Betti numbers of $I^r$ that follow from the main result.  Finally,
\cref{s:min} introduces extremal ideals, which have maximal Betti
numbers among powers of square-free monomial ideals. In particular,
\cref{p:last} provides a full characterization of the conditions on
$r$ and $q$ that guarantee $\Lrq$ supports a minimal free resolution
of $I^r$ for some ideal $I$.

This paper is an extension of the work in \cite{L2} where
the focus is on the second power of the ideal $I$.  The collaboration
was initiated at the 2019 workshop ``Women in Commutative Algebra''
hosted by the Banff International Research Station.

%%%%%%%%%%%%%%%%%%%%%%%%%%%%%%%%%%%%%%%%%%%%%%%%%%%%%%%%%%
\section{Simplicial Resolutions}
\label{sec:background}
%%%%%%%%%%%%%%%%%%%%%%%%%%%%%%%%%%%%%%%%%%%%%%%%%%%%%%%%%%

Fix $S=\sfk[x_1,\ldots,x_n]$ to be a polynomial ring over a
        field $\sfk$.  We begin by reviewing necessary background for
        simplicial complexes and simplicial resolutions and then
        demonstrate the potential relationship to resolutions of
        ideals.

A {\bf simplicial complex} $\D$ on a {\bf vertex set } $V$ is a set
of subsets of $V$ such that if $F \in \D$ and $G \subseteq F$ then $G
\in \D$.  We use the following terminology for simplicial complexes:

\begin{definition}  Let $\D$ be a simplicial complex.
\begin{enumerate}
\item An element of $\D$ is called a {\bf face}.
\item The {\bf facets} of $\D$ are the maximal faces under inclusion.
\item  The {\bf dimension} of a face $F \in \D$ is $\dim(F)=|F|-1$.
\item The {\bf dimension} of $\D$ is the maximum of the dimensions of its faces.
\item $\D$ is called a {\bf simplex} if it has one facet.
\item The {\bf $\f$-vector} $\f(\D)=(f_0,\ldots,f_d)$ of a
$d$-dimensional simplicial complex $\D$ has  $f_i=$
the number of $i$-dimensional faces of $\D$. 
\end{enumerate}
\end{definition}

Note that a simplicial complex can be uniquely determined by its facets.  One writes
$$\D = \tuple{F_1, \ldots, F_q}$$
to denote a simplicial complex $\D$ with facets $F_1,\ldots,F_q$.

In subsequent sections, we will use the tool of trimming
        simplicial complexes via certain rules.  Essentially, we will
        delete vertices in a specified fashion.  Vertex deletions
        naturally lead to the consideration of subcomplexes of
        simplicial complexes, which are defined below, together with
        additional structures and notions that will be used throughout
        the paper.

\begin{definition}
\label{d:defs} Let $\D$ be a simplicial complex on a vertex set $V$.
\begin{enumerate}
\item If $v$ is a vertex of $\D$, then the {\bf deletion} of $v$ from
  $\D$ is the simplicial complex
$$\D \setminus \{v\} = \{ \sigma \in \D \mid v \not \in \sigma\}.$$

\item A {\bf subcomplex} of $\D$ is a subset of $\D$ which is also a
  simplicial complex.

\item Given $W \subseteq V$, the {\bf induced subcomplex} of $\D$ on
  $W$ is the subcomplex
$$\D_W=\{ \sigma \in \D \mid \sigma \subseteq W\}.$$

\item A {\bf leaf}~\cite{F02} is a facet $F$ of $\D$ such that $F$ is
the only facet of $\D$, or there is a facet $G$ of $\D$ with $F\ne G$
such that
$$F\cap H \subseteq G$$ for all facets $H \neq F$. The facet $G$ is
called a {\bf joint} of $F$. (Note that the joint of a leaf need not
be unique (see \cite{F02}).

\item $\D$ is called a {\bf quasi-forest}~\cite{Z} if the facets of
$\D$ can be ordered as $F_1,\ldots, F_q$ such that for
$i=1,\ldots,q$, the facet $F_i$ is a leaf of the simplicial complex
$\tuple{F_1,\ldots,F_i}$.

\item $\D$ is a {\bf quasi-tree} if it is a connected quasi-forest.
\end{enumerate}
\end{definition}

\begin{example} \label{e:quasitrees1}
The simplicial complex below is a quasi-tree.  The leaf order is
$F_1, \ldots , F_5$, meaning that each $F_i$ is a leaf of $\langle
F_1, \dots, F_i \rangle$.  In this example the joint of $F_i$ is
$F_{i-1}$ for all $i \geq 1$. 

$$
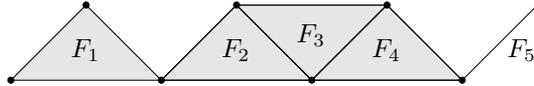
\begin{tikzpicture}
\coordinate (A) at (0, 0);
\coordinate (B) at (1, 1);
\coordinate (C) at (2, 0);
\coordinate (D) at (3, 1);
\coordinate (E) at (4, 0);
\coordinate (F) at (5, 1);
\coordinate (G) at (6, 0);
\coordinate (H) at (7, 1);
\draw  [fill=gray!20] (B) -- (C) -- (A) -- cycle;
\draw  [fill=gray!20] (D) -- (C) -- (E) -- cycle;
\draw  [fill=gray!20] (D) -- (E) -- (F) -- cycle;
\draw  [fill=gray!20] (G) -- (E) -- (F) -- cycle;
\draw[-] (C) -- (D);
\draw[-] (D) -- (E);
\draw[-] (E) -- (F);
\draw[-] (C) -- (E);
\draw[-] (F) -- (D);
\draw[-] (F) -- (G);
\draw[-] (G) -- (E);
\draw[-] (G) -- (H);
\pgfputat{\pgfxy(.8,.4)}{\pgfbox[left,center]{$F_1$}}
\pgfputat{\pgfxy(2.8,.4)}{\pgfbox[left,center]{$F_2$}}
\pgfputat{\pgfxy(3.8,.6)}{\pgfbox[left,center]{$F_3$}}
\pgfputat{\pgfxy(4.8,.4)}{\pgfbox[left,center]{$F_4$}}
\pgfputat{\pgfxy(6.6,.4)}{\pgfbox[left,center]{$F_5$}}
\draw[black, fill=black] (A) circle(0.04);
\draw[black, fill=black] (B) circle(0.04);
\draw[black, fill=black] (C) circle(0.04);
\draw[black, fill=black] (D) circle(0.04);
\draw[black, fill=black] (E) circle(0.04);
\draw[black, fill=black] (F) circle(0.04);
\draw[black, fill=black] (G) circle(0.04);
\draw[black, fill=black] (H) circle(0.04);
\end{tikzpicture}
$$
\end{example}

\begin{example} \label{e:quasitrees2}
The star-shaped complex drawn below on the left is a quasi-tree,
with leaf order $F_0, F_1, F_2, F_3$.  In particular, the center
facet $F_0$ is the joint of $F_i$ for every $i \geq 1$.  This complex
is a standard example of a quasi-tree which is not a simplicial
tree in the sense of~\cite{F02}. This particular quasi-tree is
shown in \cite{L2} to support a free resolution of the second
power of any ideal with three square-free monomial generators.

If one removes $F_0$ from the center, the remaining complex is shown
in the picture on the right. This simplicial complex is not a
quasi-tree, since no facet is a leaf.

$$
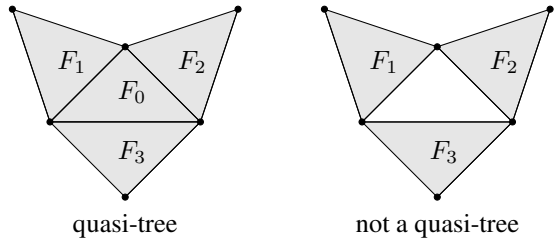
\begin{tabular}{ccc}
\begin{tikzpicture}
\tikzstyle{point}=[inner sep=0pt]
\coordinate (a) at (0,1); 
\coordinate (b) at (-1,0);
\coordinate (c) at (1,0) ;
\coordinate (d) at (-1.5,1.5);
\coordinate (e) at (1.5,1.5) ;
\coordinate (f) at (0,-1) ;
\draw [fill=gray!20](a.center) -- (b.center) -- (c.center);
\draw [fill=gray!20](a.center) -- (b.center) -- (d.center);
\draw [fill=gray!20](a.center) -- (c.center) -- (e.center);
\draw [fill=gray!20](b.center) -- (c.center) -- (f.center);
\draw (a.center) -- (b.center);
\draw (a.center) -- (c.center);
\draw (a.center) -- (d.center);
\draw (a.center) -- (e.center);
\draw (b.center) -- (c.center);
\draw (b.center) -- (d.center);
\draw (b.center) -- (f.center);
\draw (c.center) -- (e.center);
\draw (c.center) -- (f.center);
\pgfputat{\pgfxy(-.9,.8)}{\pgfbox[left,center]{$F_1$}}
\pgfputat{\pgfxy(.7,.8)}{\pgfbox[left,center]{$F_2$}}
\pgfputat{\pgfxy(-.1,-.4)}{\pgfbox[left,center]{$F_3$}}
\pgfputat{\pgfxy(-.1,.4)}{\pgfbox[left,center]{$F_0$}}
\draw[black, fill=black] (a) circle(0.04);
\draw[black, fill=black] (b) circle(0.04);
\draw[black, fill=black] (c) circle(0.04);
\draw[black, fill=black] (d) circle(0.04);
\draw[black, fill=black] (e) circle(0.04);
\draw[black, fill=black] (f) circle(0.04);
\end{tikzpicture}& \quad \quad & 
\begin{tikzpicture}
\tikzstyle{point}=[inner sep=0pt]
\coordinate (a) at (0,1); 
\coordinate (b) at (-1,0);
\coordinate (c) at (1,0) ;
\coordinate (d) at (-1.5,1.5);
\coordinate (e) at (1.5,1.5) ;
\coordinate (f) at (0,-1) ;
\draw [fill=gray!20](a.center) -- (b.center) -- (d.center);
\draw [fill=gray!20](a.center) -- (c.center) -- (e.center);
\draw [fill=gray!20](b.center) -- (c.center) -- (f.center);
\draw (a.center) -- (b.center);
\draw (a.center) -- (c.center);
\draw (a.center) -- (d.center);
\draw (a.center) -- (e.center);
\draw (b.center) -- (c.center);
\draw (b.center) -- (d.center);
\draw (b.center) -- (f.center);
\draw (c.center) -- (e.center);
\draw (c.center) -- (f.center);
\pgfputat{\pgfxy(-.9,.8)}{\pgfbox[left,center]{$F_1$}}
\pgfputat{\pgfxy(.7,.8)}{\pgfbox[left,center]{$F_2$}}
\pgfputat{\pgfxy(-.1,-.4)}{\pgfbox[left,center]{$F_3$}}
\draw[black, fill=black] (a) circle(0.04);
\draw[black, fill=black] (b) circle(0.04);
\draw[black, fill=black] (c) circle(0.04);
\draw[black, fill=black] (d) circle(0.04);
\draw[black, fill=black] (e) circle(0.04);
\draw[black, fill=black] (f) circle(0.04);
\end{tikzpicture}\\
\mbox{quasi-tree} & & \mbox{not a quasi-tree}\\
\end{tabular}
$$
\end{example}

A {\bf free resolution} of $I$ 
is an exact sequence of the form 
$$0 \rightarrow S^{\beta_t} \rightarrow S^{\beta_{t-1}} \rightarrow
\cdots \rightarrow S^{\beta_1} \rightarrow S^{\beta_0} \rightarrow I
\rightarrow 0,$$ where $S^{\beta_j}$ is a free $S$-module of rank
$\beta_j$ and $t \in \mathbb N$. When $\beta_j$ is the smallest
possible rank of a free module in the $j^{th}$ spot of any free
resolution of $I$ for each $j$, the resolution is {\bf minimal}.  In
this case, the numbers $\beta_j$ are invariants of $I$ and are called
the {\bf Betti numbers} of $I$.  

In the 1960s, Diana Taylor demonstrated a striking connection between a
$(q-1)$-simplex and a resolution of a monomial ideal $I$ in $S$.  If $I$
is a monomial ideal in $S$ minimally generated by monomials
$m_1,\ldots,m_q$, then $\Taylor(I)$ denotes the simplex with $q$
vertices indexed by the set $[q]=\{1,\dots,q\}$, where each vertex $i$ is labeled with one of the
monomials $m_i$, and each face $\sigma$ is labeled with the monomial
$$M_\sigma=\lcm(m_i \st i \in \sigma).$$

In her Ph.D thesis~\cite{T}, Taylor proved that the simplicial chain
complex of $\Taylor(I)$ gives rise to a multigraded free resolution of
$I$. In particular, the $i^{th}$ Betti number of $I$ is bounded above
by the number of $i$-dimensional faces of $\Taylor(I)$, which is
${q}\choose{i+1}$.  This method has been generalized so that if
$\D$ is a simplicial or cellular complex whose vertices are
labeled with the monomial generators $m_1, \ldots , m_q$ of an ideal
$I$ and whose faces are labeled with the least common multiple of the
vertex labels as above, then we say that $\D$ {\bf supports a free
resolution} of $I$ if the homogenization of the  simplicial (or cellular) chain complex of $\D$ is a multi-graded free resolution of $I$, denoted by $\mathbb F_{\D}$, see \cite{BPS, BS}. The multi-graded complex $\mathbb F_{\D}$ is described as follows. For each $t\ge 0$, the free module $(\mathbb F_{\D})_t$ has basis elements denoted by $e_\sigma$, where $\sigma$ ranges over all faces of  $\D$ with $|\sigma|=t+1$, and $e_{\sigma}$ is considered to have multi-degree $M_{\sigma}$. The differential is described by
\begin{equation}
\label{e:Taylor-diff}
\partial(e_\sigma)=\sum_{j=0}^t(-1)^j\frac{M_{\sigma}}{M_{\sigma\smallsetminus\{v_{i_j}\}}} e_{\sigma\smallsetminus\{v_{i_j}\}}\,,
\end{equation}
where $\sigma=\{v_{i_0}, \dots, v_{i_t}\}$ with $i_0 < i_1 < \cdots < i_t$.

\begin{example}\label{e:homogenize} 

Let $I = (xy, yz, zu)$ in $R=\sfk[x,y,z,u]$. The labeled simplicial chain complex 
$$\begin{tikzpicture}
\tikzstyle{point}=[inner sep=0pt]
\node (a)[point,label=left:$xy$] at (-2,0) {\tiny{$\bullet$}};
\node (b)[point,label=above:$yz$] at (0,0) {\tiny{$\bullet$}};
\node (c)[point,label=right:$zu$] at (2,0) {\tiny{$\bullet$}};

\draw (a.center) -- (b.center);
\draw (b.center) -- (c.center);

\pgfputat{\pgfxy(-1.2,.15)}{\pgfbox[left,center]{$xyz$}}
\pgfputat{\pgfxy(.8,.15)}{\pgfbox[left,center]{$yzu$}}

\end{tikzpicture}
$$
supports a free resolution of $I$. The chain complex of $\D$ is 
$$ 0 \longrightarrow \sfk^2 \xrightarrow{\scriptsize \begin{bmatrix}1&0\\-1&1\\0&-1\end{bmatrix}}\sfk^3 \rightarrow \sfk$$
and homogenization results in the free resolution 
$$ 0 \longrightarrow R(xyz) \oplus R(yzu) \xrightarrow{\scriptsize \begin{bmatrix}z&0\\-x&u\\0&-y\end{bmatrix}}\ R(xy) \oplus R(yz) \oplus R(zu) \longrightarrow I  \longrightarrow  0\,,$$
where the notation $R(x^ay^bz^cu^d)$ refers to the $R$-free module with one generator in multi-degree $(a,b,c,d)$. 
\end{example}

\begin{remark} \label{r:larger-Taylor}
Simplicial complexes that support a
free resolution of a monomial ideal are usually constructed such that the
vertices correspond to and are labeled by a {\it minimal} set of
generators of the ideal. However, one can also work with non-minimal
generators, at the expense of producing a larger complex. In
particular, one can mimic the construction of $\Taylor(I)$, but use
instead any set of monomial generators of $I$. The same
considerations show that this complex supports a free resolution of
$I$.
\end{remark}

There are various combinatorial ways to build subcomplexes
        $\D$ of $\Taylor(I)$ that support a free resolution of
        $I$. One such well known complex is the {\it Lyubeznik
          complex}, which supports the {\it Lyubeznik resolution} of
        $I$ (\cite{L}). The Lyubeznik resolution is the main
        inspiration for the complexes $\Lrq$ and $\Lr(I)$ which appear
        later in this paper.  We will not define this resolution since
        it is not used in this paper, but refer the reader to \cite{L,M}
        for additional information.

%%%%%%%%%%%%%%%%%%%%%%%%%%%%%%%%%%%%%%%%%%%%%%%%%%%%%%%%%%%%%%%%%%%
\section{Resolutions supported on quasi-trees}
\label{sec:res-by-qtrees}
%%%%%%%%%%%%%%%%%%%%%%%%%%%%%%%%%%%%%%%%%%%%%%%%%%%%%%%%%%%%%%%
 Taylor's resolution is usually far from
        minimal. However, for a given monomial ideal $I$, a
        criterion of Bayer, Peeva and Sturmfels (see \cref{t:BPS})
        allows one to check if a subcomplex of $\Taylor(I)$ supports a
        free resolution of $I$. In this section, using the above
        criterion and by observing that quasi-trees are collapsible,
        we show that a quasi-tree $\D$ supports a free resolution
        of a given monomial ideal if and only if certain subcomplexes
        of $\D$ are connected.

        For a subcomplex $\D$ of $\Taylor(I)$ and a
        monomial $M$ in $S$, let $\D_M$ be the subcomplex of $\D$
        induced on the vertices of $\D$ whose labels divide $M$.

        The following is the criterion of Bayer, Peeva and Sturmfels
        \cite[Lemma 2.2]{BPS}; see also \cite[Theorem 2.2]{E} for the
        statement on minimality.

\begin{theorem}[{\bf Criterion for a simplicial complex to support a free  resolution}]\label{t:BPS}
          Let $\D$ be a simplicial complex whose vertices are labeled
          with a monomial generating set of a monomial ideal
          $I$  in a polynomial ring $S$ over a field.
          Then $\D$ supports a free resolution of $I$ over $S$ if and
          only if for every monomial $M$, the induced subcomplex
          $\D_M$ of $\D$ on the vertices whose labels divide $M$ is
          empty or acyclic.

Furthermore, a free resolution supported on $\D$ is
                minimal if and only if $M_{\sigma} \neq M_{\sigma'}$
                for every proper subface $\sigma'$ of a face $\sigma$
                of $\D$.
\end{theorem} 

\begin{remark}
The results of \cref{t:BPS} are usually stated with the assumption that one uses the {\it minimal} monomial generating set of $I$ for the labels.  However, the proof of \cite[Lemma 2.2]{BPS} does not make use of this assumption, hence we formulated the result above to reflect this observation. 
\end{remark}

In particular, \cref{t:BPS}  implies that the $\f$-vector of a
complex $\D$ supporting a resolution of a monomial ideal $I$ is an
upper bound for the vector of Betti numbers of $I$. In other words,
for each $i \leq d=\dim(\D)$,
$$\beta_i(I)\leq f_{i} \qwhere \f(\D)=(f_0,\ldots,f_d).$$
In
particular, if $\D$ supports a minimal free resolution of $I$, then equality holds above. 

Using \cref{t:BPS} it is straightforward to see that to
determine whether $\D$ supports a free resolution of $I$, it suffices
to check that $\D_M$ is empty or acyclic only for monomials $M$ in the
lcm lattice of $I$; that is, for monomials $M$ that are least common
multiples of sets of vertex labels.

If the complex $\D$ under consideration in \cref{t:BPS} is a
simplicial tree, then it suffices to show that $\D_M$ is connected, instead of acyclic, see ~\cite{F14}.  More precisely, it is established in~\cite{F14} that
every induced subcomplex of a simplicial tree is  a simplicial
forest, and then it is shown that simplicial trees are acyclic, and hence an induced subcomplex of $\D$ is acyclic if and only if  it is empty or connected (see~\cite[Theorems 2.5, 2.9, 3.2]{F14}).

We now generalize the work in~\cite{F14} by showing that the criterion
in \cref{t:BPS} can be extended to the class of quasi-trees. To do so
we need to argue that quasi-trees, and their connected induced
subcomplexes, are acyclic. We do so using the following series of
results.

\begin{proposition}[{\bf Induced subcomplexes of  quasi-forests are quasi-forests}]\label{p:qtree-induced} If a simplicial complex  $\D$ is a 
quasi-forest, then every induced subcomplex of $\D$ is a quasi-forest.
\end{proposition}

\begin{proof} By~\cite[Proposition 6]{FH},
a simplicial complex $\D$ with vertex set $V$ is a quasi-forest if
and only if for every subset $W \subseteq V$, the induced subcomplex
$\D_W$ has a leaf.  If $W \subseteq V$, consider the induced
subcomplex $\D_W$ of $\D$. If $U \subseteq W$, then $\D_U=(\D_W)_U$
has a leaf, and hence, $\D_W$ is a quasi-forest.
\end{proof}

A face $\sigma$ of a simplicial complex $\D$ is called a {\bf free
face} if it is properly contained in a unique facet $F$ of
$\D$. A {\bf collapse} of $\D$ along the free face $\sigma$ is the simplicial complex obtained by removing the faces $\tau$ such that $\sigma\subseteq \tau
\subseteq F$ from $\D$.  If additionally $\dim(\sigma)=\dim(F)-1$,
then the collapse is called an {\bf elementary collapse}. The
simplicial complex $\D$ is called {\bf collapsible} if it can be
reduced to a point via a series of (elementary) collapses.

\begin{example}\label{e:L2-picture-collapsing} 
Consider the quasi-tree in \cref{e:quasitrees2}. To illustrate that
this complex is collapsible to a point, we label in
\cref{f:collapsing} the vertices and then demonstrate one collapsing
sequence. At each step, the two faces that play the roles of $\sigma$
and $F$ in the exposition above are indicated. Note that the first two
steps are elementary collapses for demonstration purposes. Alternate
collapsing sequences exist.

\begin{figure} 
$$
\begin{array}{ccc}
\begin{tikzpicture}
\tikzstyle{point}=[inner sep=0pt]
\node (a)[point,label=above:$v_1$] at (0,1) {};
\node (b)[point,label=left:$v_2$] at (-1,0) {};
\node (c)[point,label=right:$v_3$] at (1,0) {};
\node (d)[point,label=left:$v_4$] at (-1.5,1.5) {};
\node (e)[point,label=right:$v_5$] at (1.5,1.5) {};
\node (f)[point,label=below:$v_6$] at (0,-1) {};
\node (g)[label=above: $\xrightarrow{\tiny\begin{array}{c}
\sigma_1 = \{v_1, v_4\} \\F_1 =\{v_1, v_2, v_4\}
\end{array}}$] at (3,-.5) {};
\draw [fill=gray!20](a.center) -- (b.center) -- (c.center);
\draw [fill=gray!20](a.center) -- (b.center) -- (d.center);
\draw [fill=gray!20](a.center) -- (c.center) -- (e.center);
\draw [fill=gray!20](b.center) -- (c.center) -- (f.center);
\draw (a.center) -- (b.center);
\draw (a.center) -- (c.center);
\draw (a.center) -- (d.center);
\draw (a.center) -- (e.center);
\draw (b.center) -- (c.center);
\draw (b.center) -- (d.center);
\draw (b.center) -- (f.center);
\draw (c.center) -- (e.center);
\draw (c.center) -- (f.center);
 \filldraw [black] (a.center) circle (1pt);
 \filldraw [black] (b.center) circle (1pt);
 \filldraw [black] (c.center) circle (1pt);
 \filldraw [black] (d.center) circle (1pt);
 \filldraw [black] (e.center) circle (1pt);
 \filldraw [black] (f.center) circle (1pt);
\end{tikzpicture}

&
\begin{tikzpicture}
\tikzstyle{point}=[inner sep=0pt]
\node (a)[point,label=above:$v_1$] at (0,1) {};
\node (b)[point,label=left:$v_2$] at (-1,0) {};
\node (c)[point,label=right:$v_3$] at (1,0) {};
\node (d)[point,label=left:$v_4$] at (-1.5,1.5) {};
\node (e)[point,label=right:$v_5$] at (1.5,1.5) {};
\node (f)[point,label=below:$v_6$] at (0,-1) {};
\node (g)[label=above:
$\xrightarrow{\tiny
\begin{array}{c}
\sigma_2 = \{v_4\} \\
F_2 = \{v_2 , v_4 \}
\end{array}}$] at (3,-.5) {};
\draw [fill=gray!20](a.center) -- (b.center) -- (c.center);
\draw [fill=gray!20](a.center) -- (c.center) -- (e.center);
\draw [fill=gray!20](b.center) -- (c.center) -- (f.center);
\draw (a.center) -- (b.center);
\draw (a.center) -- (c.center);
\draw (a.center) -- (e.center);
\draw (b.center) -- (c.center);
\draw (b.center) -- (d.center);
\draw (b.center) -- (f.center);
\draw (c.center) -- (e.center);
\draw (c.center) -- (f.center);
 \filldraw [black] (a.center) circle (1pt);
 \filldraw [black] (b.center) circle (1pt);
 \filldraw [black] (c.center) circle (1pt);
 \filldraw [black] (d.center) circle (1pt);
 \filldraw [black] (e.center) circle (1pt);
 \filldraw [black] (f.center) circle (1pt);
\end{tikzpicture} 
&
\\

\begin{tikzpicture}
\tikzstyle{point}=[inner sep=0pt]
\node (a)[point,label=above:$v_1$] at (0,1) {};
\node (b)[point,label=left:$v_2$] at (-1,0) {};
\node (c)[point,label=right:$v_3$] at (1,0) {};
\node (e)[point,label=right:$v_5$] at (1.5,1.5) {};
\node (f)[point,label=below:$v_6$] at (0,-1) {};
\node (g)[label=above:
$\xrightarrow{\tiny\begin{array}{c}
\sigma_3= \{ v_5\}\\ F_3=\{v_1, v_3, v_5 \}
\end{array}}$] at (3,-.5) {};
\draw [fill=gray!20](a.center) -- (b.center) -- (c.center);
\draw [fill=gray!20](a.center) -- (c.center) -- (e.center);
\draw [fill=gray!20](b.center) -- (c.center) -- (f.center);
\draw (a.center) -- (b.center);
\draw (a.center) -- (c.center);
\draw (a.center) -- (e.center);
\draw (b.center) -- (c.center);
\draw (b.center) -- (f.center);
\draw (c.center) -- (f.center);
 \filldraw [black] (a.center) circle (1pt);
 \filldraw [black] (b.center) circle (1pt);
 \filldraw [black] (c.center) circle (1pt);
 \filldraw [black] (e.center) circle (1pt);
 \filldraw [black] (f.center) circle (1pt);
\end{tikzpicture} 

&
\begin{tikzpicture}
\tikzstyle{point}=[inner sep=0pt]
\node (a)[point,label=above:$v_1$] at (0,1) {};
\node (b)[point,label=left:$v_2$] at (-1,0) {};
\node (c)[point,label=right:$v_3$] at (1,0) {};
\node (f)[point,label=below:$v_6$] at (0,-1) {};
\node (g)[label=above:
$\xrightarrow{\tiny\begin{array}{c}
\sigma_4= \{ v_6\}\\ F_4=\{v_2, v_3, v_6 \}
\end{array}}$] at (3,-.5) {};
\draw [fill=gray!20](a.center) -- (b.center) -- (c.center);
\draw [fill=gray!20](b.center) -- (c.center) -- (f.center);
\draw (a.center) -- (b.center);
\draw (a.center) -- (c.center);
\draw (b.center) -- (f.center);
\draw (c.center) -- (f.center);
 \filldraw [black] (a.center) circle (1pt);
 \filldraw [black] (b.center) circle (1pt);
 \filldraw [black] (c.center) circle (1pt);
 \filldraw [black] (f.center) circle (1pt);
\end{tikzpicture} 

&
\\
\begin{tikzpicture}
\tikzstyle{point}=[inner sep=0pt]
\node (a)[point,label=above:$v_1$] at (0,1) {};
\node (b)[point,label=left:$v_2$] at (-1,0) {};
\node (c)[point,label=right:$v_3$] at (1,0) {};
\node (g)[label=above: $\xrightarrow{\tiny\begin{array}{c}
\sigma_5 = \{v_1\} \\ F_5 = \{v_1, v_2, v_3 \}
\end{array}}$] at (3,-.5) {};
\draw [fill=gray!20](a.center) -- (b.center) -- (c.center);
\draw (a.center) -- (c.center);

 \filldraw [black] (a.center) circle (1pt);
 \filldraw [black] (b.center) circle (1pt);
 \filldraw [black] (c.center) circle (1pt);
\end{tikzpicture} 

& \begin{tikzpicture}
\tikzstyle{point}=[inner sep=0pt]
\node (b)[point,label=left:$v_2$] at (-1,0) {};
\node (c)[point,label=right:$v_3$] at (1,0) {};
\node (g)[label=above: $\xrightarrow{\tiny\begin{array}{c}
\sigma_6 = \{v_2\} \\ F_6 = \{v_2, v_3 \}
\end{array}}$] at (3,-.5) {};
\draw (b.center) -- (c.center) ;
 \filldraw [black] (b.center) circle (1pt);
 \filldraw [black] (c.center) circle (1pt);
\end{tikzpicture} &

\begin{tikzpicture}
\tikzstyle{point}=[inner sep=0pt]
\node (b)[point,label=right:$v_3$] at (-1,0) {};
\node (c) at (1,0) {};
\node (g) at (3,-.5) {};
 \filldraw [black] (b.center) circle (1pt); 
\end{tikzpicture}
\end{array} 
$$
\caption{An example of collapsing}
\label{f:collapsing}
\end{figure}
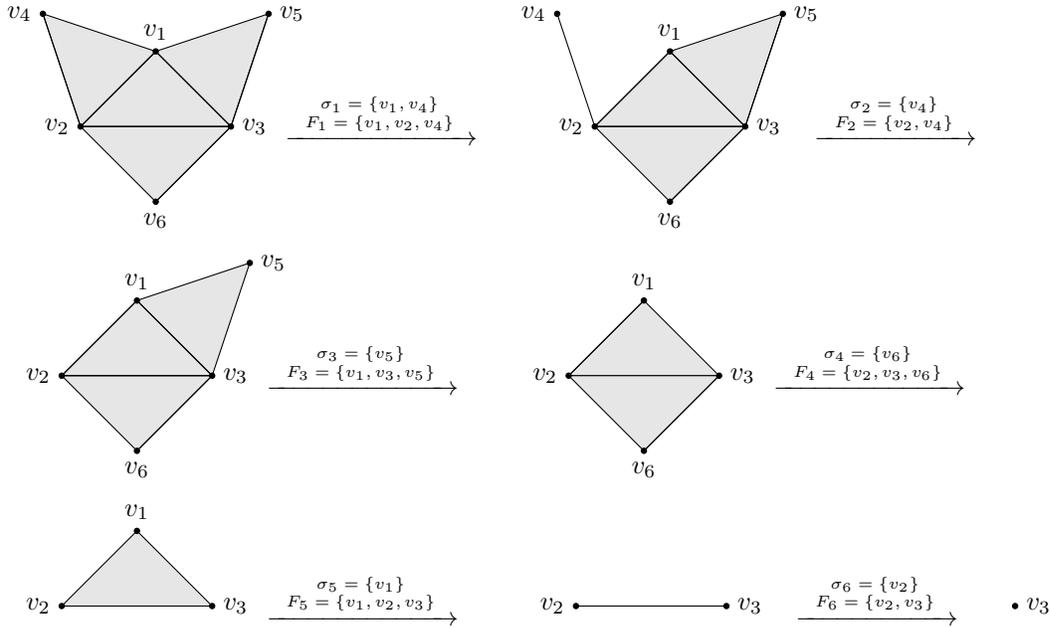

\end{example} 

In general, by starting with a leaf and the face that contains it, any
quasi-tree can be collapsed to a point.

\begin{proposition}\label{p:qtree-collapsible}
Every quasi-tree is collapsible. 
\end{proposition}

\begin{proof} 
Proceed by induction on $q$, the number of facets of $\D$. When $q=1$,
$\D$ is a simplex, and known to be collapsible (e.g. \cite[Proposition~2.7]{F14}). Suppose the statement holds for all quasi-trees with
fewer than $q$ facets.

Let $\D$ be a quasi-tree with facet ordering $F_1,\ldots,F_q$ so
that,  for $i=1,\ldots,q$, each $F_i$ is a leaf of $\tuple{F_1,\ldots,F_i}$.  Since $F_q$ is a leaf of $\D$, it intersects
$\D'=\tuple{F_1,\ldots,F_{q-1}}$ in a face $F'$ of $\D'$. By~\cite[Proposition~2.7]{F14}, $F_q$ collapses down to the face
$F'$ by removing faces outside $\D'$. As a result, this series of
collapses brings $\D$ to the quasi-tree $\D'$, which, by the induction
hypothesis, is collapsible. 
\end{proof}

We now use \cref{p:qtree-collapsible} and \cref{p:qtree-induced} to
show that the Bayer-Peeva-Sturmfels criterion in \cref{t:BPS} can be extended to the class
of quasi-trees.

\begin{theorem}[{\bf Criterion for a quasi-tree to support a free  resolution}]\label{t:res-by-quasitree} Let $\D$ be a quasi-tree whose vertices are labeled
with a monomial generating set of a monomial ideal $I$ in a
polynomial ring $S$ over a field.  Then $\D$ supports a free
resolution of $I$ over $S$ if and only if for every monomial $M$, the
induced subcomplex $\D_M$ of $\D$ on the vertices whose labels divide
$M$ is empty or connected.
\end{theorem}

\begin{proof} By \cref{t:BPS}, $\D$ supports a
resolution of $I$ if and only if $\D_M$ is empty or acyclic for
every monomial $M \in S$. If $M \in S$ is a monomial, then by
\cref{p:qtree-induced} above, $\D_M$ is a quasi-forest.  Moreover,
by \cref{p:qtree-collapsible}, every connected component of $\D_M$
is contractible, i.e., homotopy equivalent to a point, and hence
acyclic. The only possible homology that $\D_M$ could have would be that which comes
from it being disconnected. As a result, $\D$ supports a
resolution of $I$ if and only if $\D_M$ is empty or connected for
every monomial $M \in S$.
\end{proof}

%%%%%%%%%%%%%%%%%%%%%%%%%%%%%%%%%%%%%%%%%%%%%%%%%%%%%%%%%%%%%%%%%%%%%
\section{The quasi-tree $\Lrq$}
\label{sec:Lrq}
%%%%%%%%%%%%%%%%%%%%%%%%%%%%%%%%%%%%%%%%%%%%%%%%%%%%%%%%%%%%%%%%%%%%%%%

Recall that our goal is to find a simplicial complex smaller than
Taylor's complex that supports a free resolution of $I^r$ when $I$
has $q$ square-free monomial generators. In this section we construct
a subcomplex $\Lrq$ of the Taylor simplex, which depends only on the
integers $r$ and $q$, and contains a subcomplex supporting a free
resolution of $I^r$.

The base case for our construction is the case $r=1$. In this case
$\LL^1_q$ is the well-known Taylor complex~\cite{T}: a simplex with $q$ vertices
that supports a free resolution of any monomial ideal with
$q$ generators.  The case $r=2$ was investigated in the earlier work
of the authors~\cite{L2}. For example, it was shown in~\cite{L2}
that $\LL^2_3$ is the quasi-tree in \cref{e:quasitrees2}.

We now collect and formalize the notation needed for the extension to
the case $r \geq 3$.

\begin{notation} \label{n:rqs}
Let $r$ and $q$ be two positive integers. 

\begin{itemize}
\setlength\itemsep{1em}

\item Let $\be_1, \ldots, \be_q$ denote the standard basis vectors for
$\mathbb R^q$; i.e., for each $i\in [q]$, $\be_i$ is the $q$-tuple with 1 in the $i^{th}$
coordinate and 0 elsewhere.

\item Define $\Nrq$ to be the set of points in $ \ZZ_{\ge 0}^q$ whose
coordinates add up to $r$:
\begin{align*}
\Nrq = &\{(a_1,\ldots, a_q) \in \ZZ_{\ge 0}^q \st a_1 + \cdots +
a_q =r\}\\ =&\{a_1\be_1 + \cdots +a_q\be_q \st a_i \in \ZZ_{\ge 0}
{\text{ and }} a_1 + \cdots + a_q =r\}.
\end{align*}

\item Set $s=\left\lceil{\frac{r}{2}}\right\rceil$; that is, when $r$ is odd,
$r=2s-1$, and when $r$ is even, $r=2s$. 
\end{itemize}
\end{notation}

\begin{definition}[{\bf The simplicial complex $\Lrq$ - see also \cref{p:Lrq}}]\label{d:Lrq}
Let $r, q\ge 1$ be two integers. Following \cref{n:rqs} we define
$\Lrq$ to be the simplicial complex with vertex set $\Nrq$ whose faces
are all subsets of the (not necessarily distinct) sets $F^r_1, \dots,
F^r_{q}, G^r_1, \ldots , G^r_q$ defined as
\begin{align*}
F^r_i &=\{(a_1, \ldots, a_q) \in \Nrq \colon a_i \le \max\{r-1, s\} \mbox{ and }
a_j \leq s \mbox{ for } j \neq i\}\\ {}&\\ G^r_i&=\{(a_1, \ldots, a_q)
\in \Nrq \colon a_i \ge r-1\} = \{(r-1) \be_i + \be_j \colon j \in
    [q]\}
\end{align*}
for each $i\in [q]$.  We refer to the set $\{F^r_1, \dots, F^r_q\}$ as
the {\bf first layer} of $\Lrq$ and the set $\{G^r_1, \dots, G^r_q\}$
as the {\bf second layer} of $\Lrq$. We define the {\bf base} of
$\Lrq$ to be the face $$B^r=\left\{(a_1, \ldots, a_q) \in \Nrq \colon
a_i \le s \mbox{ for all } i\in [q] \right\},$$ so that
$$F_i^r=B^r \cup \{(a_1, \ldots, a_q) \in \Nrq \colon s+1 \le a_i \le r-1\}.$$
\end{definition}

In general, $F^r_1, \ldots F^r_q, G^r_1, \ldots , G^r_q$ are the
facets of $\Lrq$ (\cref{p:Lrq}); however, for small values of $r$
and $q$, these sets need not be distinct. We summarize these facts
in \cref{p:Lrq} to give a more precise description of the facets of
$\Lrq$.

\begin{proposition}[{\bf Equivalent definition of $\Lrq$}] 
\label{p:Lrq}
The simplicial complex $\Lrq$ in \cref{d:Lrq} can be described in
terms of its distinct facets:
$$
\Lrq=\begin{cases}
\langle F^r_1, F^r_2, \dots, F^r_q, G^r_1,\dots, G^r_q\rangle &
\text{if $r>3$ and $q\ge 2$}\\
\langle B^r, G^r_1,\dots, G^r_q\rangle &
\text{if $r = 3$ and $q\geq 2$ or $r = 2$ and $q > 2$ }\\
\langle G^2_1, G^2_2\rangle &
\text{if $r = 2$ and $q=2$}\\
\langle \N^r_q \rangle &
\text{if $r=1$ or $q=1$.}\\
\end{cases}
$$  
\end{proposition}

\begin{proof}

When $r=1$, then $s = 1$ and $\N^1_q = \{\be_1, \ldots, \be_q\}$.
It follows that $$B^1 = F^1_i=G^1_i=\N^1_q \qforall i\in [q].$$

In this case $\LL^1_q$ is a simplex with $q$ vertices (the
Taylor simplex).

If $q=1$ and $r>1$, then $$B^r=F^r_1=\emptyset \qand
G^r_1=\N^r_1 = \{(r)\}.$$

If $q = r = 2$, then
$$G^2_1 = \{(2,0),(1,1)\} \qand G^2_2 = \{(0,2),(1,1)\}$$ and since
$s=1=r-1$,
$$F^2_1=F^2_2=B^2=\{(1,1)\} \subseteq G^2_1\cap G^2_2.$$

When $q\ge2$ and $r=3$, or $q >2$ and $r = 2$, then the $G_i^r$
are distinct as each $G_i^r$ is the unique facet containing the vertex
$r\be_i$.  Furthermore, since $s+1=r$,
$$F^r_1=\dots=F^r_q=B^r$$ contains all the vertices $(r-1)\be_i+\be_j$
where $i\neq j$, and so $B^r$ is not embedded in any of the $G_i^r$.

Finally, when $q\ge2$ and $r>3$, then $s <
                r-1$. Therefore, the $G_i^r$ are distinct, since each
                $G_i^r$ is the unique facet containing the vertex
                $r\be_i$.

For $i, j \in [q]$, $i\neq j$, $F^r_j$ is not
                contained in $F^r_i$ since
$$(s+1) \be_j + (r - s-1) \be_i \in F^r_j \setminus F^r_i.$$

Moreover, $r \be_j \notin F^r_i$, showing that $$G^r_j \nsubseteq F^r_i.$$

Lastly, no $F^r_i$ can be contained in $G^r_j$ since $B^r \subseteq F^r_i$,
but $B^r \cap G^r_j = \emptyset$ when $s < r-1$.
\end{proof}

\begin{example}\label{e:q=2}
The geometric realization of the simplicial complex $\LL^3_2$
is a path of length~3, as can be seen in the figure below.

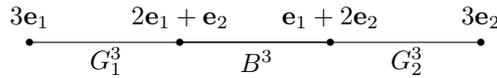
\begin{center}
\begin{tikzpicture}
\tikzstyle{point}=[inner sep=0pt]
\node (a)[point,label=above:$3\be_1$] at (-3,0) {\tiny{$\bullet$}};
\node (b)[point,label=above:$2\be_1+\be_2$] at (-1,0) {\tiny{$\bullet$}};
\node (c)[point,label=above:$\be_1+2\be_2$] at (1,0) {\tiny{$\bullet$}};
\node (d)[point,label=above:$3\be_2$] at (3,0) {\tiny{$\bullet$}};
\draw (a.center) -- (d.center);
\draw (b.center) -- (c.center);
\pgfputat{\pgfxy(-2.2, -.25)}{\pgfbox[left,center]{$G^3_1$}}
\pgfputat{\pgfxy(-.2,-.25)}{\pgfbox[left,center]{$B^3$}}
\pgfputat{\pgfxy(1.8,-.25)}{\pgfbox[left,center]{$G^3_2$}}
\end{tikzpicture}
\end{center}
Since $$\N^3_2 =\{(3,0),(2,1),(1,2),(0,3)\} = \{3\be_1, 2\be_1+\be_2, \be_1+2\be_2, 3\be_2\},$$ according to \cref{p:Lrq}, the facets of
$\LL^3_2$ are
\begin{align*}
B^3 &= \{(2,1),(1,2)\} = \{2\be_1 + \be_2, \be_1+ 2\be_2\}\\
G^3_1 &= \{(3,0),(2,1)\} = \{3\be_1, 2\be_1 + \be_2\}\\
G^3_2 &= \{(0,3),(1,2)\} = \{3\be_2, \be_1 + 2\be_2\}.
\end{align*}

By contrast, if $q=2$ and $r \geq 4$ is even, then
$B^r$ is a single vertex
$$B^r = \{(r/2)\be_1 + (r/2)\be_2 \} \subsetneq F^r_i \qforall i \in
[q].$$ For instance, if $r=6$, then $\LL^6_2$ is pictured below, and
$B^6$ is the single point at the middle of the `bow-tie' and only the
$F^r_i$s and $G^r_j$s are facets.

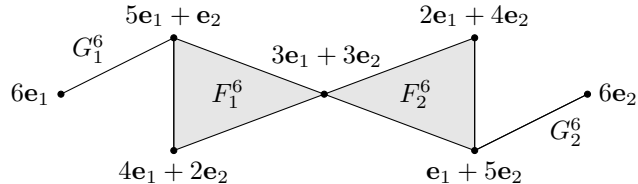
\begin{center}
\begin{tikzpicture}
\tikzstyle{point}=[inner sep=0pt]
\node (a)[point,label=left:$6\be_1$] at (-3.5,0)  {};
\node (b)[point,label=above:$5\be_1+\be_2$] at (-2,.75)  {};
\node (c)[point,label=below:$4\be_1+2\be_2$] at (-2,-.75)   {};
\node (d)[point] at (0,0)   {};
\node (e)[point,label=above:$2\be_1+4\be_2$] at (2,.75)  {};
\node (f)[point,label=below:$\be_1+5\be_2$] at (2,-.75)   {};
\node (g)[point,label=right:$6\be_2$] at (3.5,0)   {};
\draw [fill=gray!20](d.center) -- (e.center) -- (f.center);
\draw [fill=gray!20](b.center) -- (c.center) -- (d.center);
\draw (a.center) -- (b.center);
\draw (f.center) -- (g.center);
\draw (b.center) -- (f.center);
\draw (b.center) -- (c.center);
\draw (e.center) -- (f.center);
\draw (f.center) -- (g.center);
\pgfputat{\pgfxy(-3.35,.6)}{\pgfbox[left,center]{$G^6_1$}}
\pgfputat{\pgfxy(-1.5, 0)}{\pgfbox[left,center]{$F^6_1$}}
\pgfputat{\pgfxy(1,0)}{\pgfbox[left,center]{$F^6_2$}}
\pgfputat{\pgfxy(3,-.5)}{\pgfbox[left,center]{$G^6_2$}}
\pgfputat{\pgfxy(-.7,.5)}{\pgfbox[left,center]{$3\be_1+3\be_2$}}
\draw[black, fill=black] (a) circle(0.04);
\draw[black, fill=black] (b) circle(0.04);
\draw[black, fill=black] (c) circle(0.04);
\draw[black, fill=black] (d) circle(0.04);
\draw[black, fill=black] (e) circle(0.04);
\draw[black, fill=black] (f) circle(0.04);
\draw[black, fill=black] (g) circle(0.04);
\end{tikzpicture}
\end{center} 

\end{example}

Note that the points in $\N_q^r$ can be viewed as lattice points in $\mathbb{R}^q$. Indeed, they are precisely the integer lattice points in a hyperplane section of the first octant, cut out by the hyperplane whose equation is 
	$$x_1+\cdots +x_q=r.$$
While for small values of $q$ this gives a natural way to depict the points, it does not illustrate the simplicial structure well. For instance, in \cref{e:q=2}, the $6$ points would be co-linear, while our depiction emphasizes the existence of the two triangles. Using the combinatorial approach rather than the embedding in $\mathbb{R}^q$ also allows for a generalized depiction, seen in the following example, that can easily be extended to higher $q$ and $r$.

%%%%%%%%%%%%%%%%%%%%%%%%%%%%%%%%%%%%%%%%%%%%

\begin{example}\label{e:q=3} We examine the case $q=3$ in the same manner as above. The simplicial complexes
$\LL^1_3$, $\LL^2_3$ and $\LL^3_3$ appear in \cref{f:33}, and
  $\LL^6_3$ appears in \cref{f:L63}.

\def\firstcircle{(0,0) circle (2cm)}
        \def\secondcircle{(-2,0.5) circle (1cm)}
        \def\thirdcircle{(0, -2) circle (1cm)}
\def\fourthcircle{(1.3, 1.5) circle (1cm)}

%%%%%%%%%%%%%%%%%%%%%%%%%%%%%%%%%%%%%%%%%%%%
\begin{figure}
\begin{tabular}{ccc}
\begin{tikzpicture}
 \tikzstyle{point}=[inner sep=0pt]
 \node (a)[] at (0,1) {\tiny{$\bullet$}};
 \node (b)[] at (-1,0) {\tiny{$\bullet$}};
 \node (c)[] at (1,0) {\tiny{$\bullet$}};
 \node (d)[] at (-1.5,1.5) { };
 \node (e)[] at (1.5,1.5) { };
 \node (f)[] at (0,-1) {};
 \draw [fill=gray!20](a.center) -- (b.center) -- (c.center);
 \draw (a.center) -- (b.center);
 \draw (a.center) -- (c.center);
 \draw (b.center) -- (c.center);
 \filldraw [black] (a.center) circle (1pt);
 \filldraw [black] (b.center) circle (1pt);
 \filldraw [black] (c.center) circle (1pt);
 \pgfputat{\pgfxy(-.1,.4)}{\pgfbox[left,center]{$B^1$}}
\end{tikzpicture}&
\begin{tikzpicture}
\tikzstyle{point}=[inner sep=0pt]
 \node (a)[] at (0,1) {\tiny{$\bullet$}};
 \node (b)[] at (-1,0) {\tiny{$\bullet$}};
 \node (c)[] at (1,0) {\tiny{$\bullet$}};
 \node (d)[] at (-1.5,1.5) {\tiny{$\bullet$}};
 \node (e)[] at (1.5,1.5) {\tiny{$\bullet$}};
 \node (f)[] at (0,-1) {\tiny{$\bullet$}};
 \draw [fill=gray!20](a.center) -- (b.center) -- (c.center);
 \draw [fill=gray!20](a.center) -- (b.center) -- (d.center);
 \draw [fill=gray!20](a.center) -- (c.center) -- (e.center);
 \draw [fill=gray!20](b.center) -- (c.center) -- (f.center);
 \draw (a.center) -- (b.center);
 \draw (a.center) -- (c.center);
 \draw (a.center) -- (d.center);
 \draw (a.center) -- (e.center);
 \draw (b.center) -- (c.center);
 \draw (b.center) -- (d.center);
 \draw (b.center) -- (f.center);
 \draw (c.center) -- (e.center);
 \draw (c.center) -- (f.center);
 \filldraw [black] (a.center) circle (1pt);
 \filldraw [black] (b.center) circle (1pt);
 \filldraw [black] (c.center) circle (1pt);
 \filldraw [black] (d.center) circle (1pt);
 \filldraw [black] (e.center) circle (1pt);
 \filldraw [black] (f.center) circle (1pt);
\pgfputat{\pgfxy(-.9,.8)}{\pgfbox[left,center]{$G^2_1$}}
\pgfputat{\pgfxy(.7,.8)}{\pgfbox[left,center]{$G^2_2$}}
\pgfputat{\pgfxy(-.1,-.4)}{\pgfbox[left,center]{$G^2_3$}}
\pgfputat{\pgfxy(-.1,.4)}{\pgfbox[left,center]{$B^2$}}
\end{tikzpicture}&  
\begin{tikzpicture}
 \draw \firstcircle node[below ,yshift=-2mm, xshift=-3mm  ] {\tiny{$ \begin{array}{cc} \bullet\\ \be_1+\be_2+\be_3 \end{array}$}};
 \draw \firstcircle node[above, xshift=3mm  ] {\small{$B^3$}};
 \draw \secondcircle node [below left, xshift=-3mm  ] {\small{$G^3_2 $}};
 \draw \secondcircle node [above left, xshift=-1mm ] {\tiny{$\begin{array}{cc} 3\be_2\\ \bullet \end{array}$}};
 \draw \secondcircle node  [right , xshift=-2mm  ] {\tiny{$\begin{array}{cc} \bullet \hspace{1mm} \bullet \\ 2\be_2+\be_i \\i \neq 2 \end{array}$}};
 \draw \thirdcircle node [below, yshift=-3mm, xshift=-1mm ] {\small {$G^3_3$}};
 \draw \thirdcircle node [below right] {\tiny $\begin{array}{cc} 3\be_3\\ \bullet\end{array}$};
 \draw \thirdcircle node [above] {\tiny $\begin{array}{cc} \bullet  \hspace{1mm} \bullet \\ 2\be_3+\be_i\\ i \neq 3\end{array}$};
 \draw \fourthcircle node [above,yshift=3mm  ] {\small {$G^3_1$}};
 \draw \fourthcircle node [above right] {\tiny $\begin{array}{cc} 3\be_1\\ \bullet \end{array}$};
 \draw \fourthcircle node [below, xshift=-2mm] {\tiny $\begin{array}{cc} \bullet  \hspace{1mm} \bullet \\ 2\be_1+\be_i\\ i \neq 1\end{array}$};
\end{tikzpicture}\\
&&\\
$\LL^1_3$ & $\LL^2_3$& $\LL^3_3$\\
\end{tabular}
\caption{A picture of $\LL_3^1$, $\LL_3^2$ and $\LL_3^3$}
\label{f:33}
\end{figure}

%%%%%%%%%%%%%%%%%%%%%%%%%%%%%%%%%%%%%%%%%%
\def\firstcircl{(1.25,3.25) circle (1.05cm)}
\def\secondcircl{(-2.1,0) circle (1.05cm)}
\def\thirdcircl{(4.6, 0) circle (1.05cm)}
%%%%%%%%%%%%%%%%%%%%%%%%%%%%%%%%%%%%%%%%%%

\begin{figure}
\begin{tikzpicture}
\draw  [rotate=90] (0, 0) ellipse (1cm and 2.cm)  node[xshift=-2.5mm, yshift= -2.2mm] {\tiny{$ \begin{array}{cc} \bullet  \hspace {1mm} \bullet  \hspace {1mm} \bullet \\ 4\be_3+\be_i+\be_j \\ i, j \neq 3 \end{array}$}} node[above, xshift=-2mm, yshift= 3mm]{{\small $F^6_3$}} ;  ;
\draw  [rotate=90, yshift= -2.5cm] (0, 0) ellipse (1cm and 2 cm) node[xshift=3mm, yshift= -2mm] {\tiny{$ \begin{array}{cc} \bullet  \hspace {1mm} \bullet  \hspace {1mm} \bullet \\ 4\be_2+\be_i+\be_j \\ i, j \neq 2 \end{array}$}} node[above, xshift=2mm, yshift= 3mm]{{\small $F^6_2$}} ;
\draw [ xshift= 1.25cm,yshift= 1.25cm](0, 0) ellipse (1cm and 2cm) node[xshift=2mm, yshift= 1mm] {\tiny{$ \begin{array}{cc} \bullet  \hspace {1mm} \bullet  \hspace {1mm} \bullet \\ 4\be_1+\be_i+\be_j \\ i, j \neq 1 \end{array}$}} node[above , xshift=-2mm, yshift= 4.5mm]{{\small $F^6_1$}}  node[below, yshift=-1cm]{$B^6$}  ;

\node [style={circle,fill=black!7, minimum size=3.5 em}] (B) at (1.25, 0){$B^6$};
\node (A) at (1.25, -1.5){{\tiny $B^6=F^6_1 \cap F^6_2 = F^6_1 \cap F^6_3 = F^6_2 \cap F^6_3 $}};
\draw[->, dashed] (B)--(A);

\draw \firstcircl node[above , xshift= 1mm] {\tiny{$ \begin{array}{cc} \bullet \\ 6\be_1 \end{array}$}};
\draw \firstcircl node[above , yshift= 2mm, xshift= -4mm] {\small{$ G^6_1$}};
\draw \firstcircl node[below, yshift= -1.4mm] {\tiny{$ \begin{array}{cc} \bullet  \hspace {0.5mm} \bullet  \\ 5\be_1+\be_i\\i \neq 1 \end{array}$}};
\draw \secondcircl node[below ,yshift=-0.5mm, xshift=-2.mm  ] {\tiny{$ \begin{array}{cc} \bullet \\ 6\be_3 \end{array}$}};
\draw \secondcircl node[above left] {\small{$G^6_3$}};
\draw \secondcircl node(E)[right,  xshift= 1.5mm] {\tiny{$ \begin{array}{cc} \bullet  \hspace {1.5mm} \bullet \end{array}$}};
\draw \thirdcircl node[below right ,yshift=-2mm, xshift=-3mm  ] {\tiny{$ \begin{array}{cc} \bullet \\ 6 \be_2\end{array}$}};
\draw \thirdcircl node[above right] {\small{$G^6_2$}};
\draw \thirdcircl node(C)[left, xshift= -1mm] {\tiny{$ \begin{array}{cc} \bullet  \hspace {1.5mm} \bullet  \end{array}$}};

\node (D) at (3.85, 1.5) {$ {\tiny \begin{array}{cc}  5\be_2+\be_i\\i \neq 2 \end{array}}$};
\draw[->, dashed] (C)--(D);

\node(F) at (-1.38, 1.5) {\tiny{$ \begin{array}{cc} 5\be_3+\be_i\\i \neq 3 \end{array}$}};
\draw[->, dashed] (E)--(F);
\end{tikzpicture}\caption{A picture of $\LL^6_3$}\label{f:L63}
\end{figure}
\end{example}

%%%%%%%%%%%%%%%%%%%%%%%%%%%%%%%%%%%%%%%%%%%%%%%%%%%%%%%%%%%%  

The following statement generalizes~\cite[Proposition 3.3]{L2},
which deals with the special case $r=2$.

\begin{proposition}[{\bf $\Lrq$ is a quasi-tree}]\label{p:Lr-quasitree}
The simplicial complex $\Lrq$ is a quasi-tree.
\end{proposition}

\begin{proof}    Following \cref{p:Lrq} for a description of the
facets of $\Lrq$, we consider each case separately.

When $r=1$ or $q=1$, $\LL^1_1$ is a simplex,
and hence by definition a quasi-tree.

When $r = q = 2$,  $\LL^2_2$ has only two
facets, and it is trivially a quasi-tree.

When $r=2$ and $q > 2$, then  order the
facets of $\LL^2_q$ by $$B^2, G^2_1, \ldots, G^2_q.$$ In this case,
if $i \neq j$ with $i,j \in [q]$, then
$$G^2_i \cap B^2 \subseteq B^2\qand G^2_i \cap G^2_j= \{\be_i + \be_j\}
\subseteq B^2.$$ Thus each $G^2_i$ is a leaf of $\tuple{B^2, G^2_1,
\ldots, G^2_i}$ with joint $B^2$.

When $r = 3$ and $q \geq 2$, then as in the previous case,
$$B^3, G^3_1, \ldots, G^3_q$$ is
a leaf order.  Since $G^3_i \cap G^3_j= \emptyset$ when $i
\neq j$, for each $i \in [q]$ the facet $G^3_i$ is a leaf of $$\langle B^3,
G^3_1, \ldots, G^3_i \rangle$$ with joint $B^3$.

When $r>3$ and $q \geq 2$, we claim
\begin{equation}\label{e:leaf-order}
F^r_1, F^r_2, \dots, F^r_q, G^r_1,\dots, G^r_q
\end{equation} is a leaf order for $\Lrq$, making it a quasi-tree.
To see this note that if $i\neq j$, $i,j\in [q]$, then
\begin{align*}
F^r_i\cap G^r_j& =\emptyset, \quad  F^r_i\cap F^r_j=B^r, \quad
G^r_i\cap G^r_j= \emptyset,  \\ 
F^r_i \cap G^r_i & = \{(r-1)\be_i + \be_h \colon h \neq i \}.
\end{align*}

These observations show that for each $j\in [q]$, $G^r_j$ is a leaf
of $$\langle F^r_1, \dots, F^r_q, G^r_1, \ldots, G^r_{j}\rangle$$
with joint $F^r_{j}$, and for each $j\in \{2,\ldots, q\}$ $F^r_j$
is a leaf of $$\langle F^r_1, \ldots, F^r_j\rangle$$ with
joint, say, $F^r_{1}$. Hence \eqref{e:leaf-order} is a leaf order
and $\Lrq$ is a quasi-tree.
\end{proof}

 %%%%%%%%%%%%%%%%%%%%%%%%%%%%%%%%%%%%%%%%%%%%%%%%%%%%%%%%%%%%%%%%%%%%%
\section{The quasi-tree $\Lr(I)$ supporting a free resolution of $I^r$}
\label{sec:LrI}
%%%%%%%%%%%%%%%%%%%%%%%%%%%%%%%%%%%%%%%%%%%%%%%%%%%%%%%%%%%%%%%%%%%%%%%

Given an ideal $I$ with $q$ square-free monomial generators, we now
define an induced subcomplex of $\Lrq$, denoted $\Lr(I)$, which is
obtained by deleting vertices representing redundant generators of
$I^r$.  We show in \cref{t:main} that both $\Lrq$ and $\Lr(I)$ support
a free resolution of $I^r$.

%%%%%%%%%%%%%%%%%%%%%%%%%%%%%%%%%%%%%%%%%%%%%%%%%%%%%%%%%%%%

\begin{definition}[{\bf The simplicial complex $\Lr(I)$}]\label{d:LrI}
  Let $I$ be an ideal minimally generated by monomials
  $m_1,\ldots,m_q$ in the polynomial ring $S$. For $\ba = (a_1,
  \ldots, a_q) \in \Nrq$ we set
  $$\bma={m_1}^{a_1} \cdots {m_q}^{a_q}.$$ Define a partition of $\Nrq$
  into equivalence classes $V_1,\ldots,V_t$ by
  $$\ba \sim \bb \iff \bma=\bmb.$$
 We use these equivalence classes to build an induced subcomplex $\Lr(I)$, of $\Lrq$ using the following steps:
  \begin{enumerate}

  \item[Step 1.] From each equivalence class $V_i$ pick a unique
    representative $\bc_i$ in the following way: if $V_i \cap B^r \neq \emptyset$,
    then $\bc_i \in B^r$. Otherwise, choose any $\bc_i \in V_i$.
    
  \item[Step 2.]  From the set $\{\bc_1,\ldots,\bc_t\}$, eliminate all
    $\bc_i$ for which ${\m}^{\bc_j} \mid {\m}^{\bc_i}$ for some $j \neq
    i$. We call the remaining set, without loss of generality, $\{
    \bc_1,\ldots,\bc_u\}.$
  
  \item[Step 3.] Set $V=\{\bc_1,\ldots,\bc_u\}$.

  \end{enumerate}

 Define $\Lr(I)$ to be the induced subcomplex of $\Lrq$ on the
  vertex set $V$.
  \end{definition}

  The complex $\Lr(I)$ is a subcomplex of $\Taylor(I^r)$.  While
  $\Lr(I)$ is dependent on the choices made when building its vertex
  set $V$, we will abuse the notation and ignore these choices, as
  they do not have an impact on our results.  Also note that the set
  of monomials $\{ \m^{\bc _1},\ldots,\m^{\bc _u}\}$ forms a minimal
  monomial generating set for the ideal $I^r$.  There are known
  classes of ideals, for instance, square-free monomial ideals of
  projective dimension one \cite[Proposition 4.1]{morse}, for which
  the generating set $\{ \bma \mid \ba \in \Nrq\}$ does not contain
  redundancies, in which case $\Lrq = \Lr(I)$. However, in general,
  $\Lr(I)$ will be a proper subcomplex of $\Lrq$. Information about
  known redundancies in $\{ \bma \mid \ba \in \Nrq\}$ can be used to
  produce $\Lr(I)$. As an example, consider edge ideals of graphs. For
  such ideals, the redundancies in $\{ \bma \mid \ba \in \Nrq\}$ are
  encoded in the ideal of equations of the fiber cone of the ideal,
  whose generators correspond to primitive even closed walks in the
  graph (see \cite{villa} for relevant definitions and details).

\begin{example}  \label{e:square}  
  Let $S=\sfk[x,y,z,w]$ and $I = (xy,yz, zw, xw)=(m_1,m_2,m_3,m_4)$.
By \cref{p:Lrq}, the facets of $\LL_4^2$ are a $5$-simplex $B^2$ 
and four tetrahedra $G_i^2$ for $1 \leq i \leq 4$, depicted on the left in the figure below.

By \cref{d:LrI}, since $m_1m_3=m_2m_4$ is the only equation determining an equivalence class with more than one element,
we select the vertex $\be_1+\be_3$ to represent this equivalence class and form $\LL^2(I)$.
Then $\LL^2(I)$ consists of a $4$-simplex on the vertices 
$$\be_1+\be_2,\be_1+\be_3,\be_1+\be_4,\be_2+\be_3,\be_3+\be_4$$
together with two triangles and two tetrahedra depicted on the right
in the figure below. Note that vertex $\be_2+\be_4$ has been removed.
The edges depicted by dotted lines
in the figure of $\LL_4^2$ do not appear in $\LL^2(I)$ and higher
dimensional faces of $\LL_4^2$ containing $\be_2+\be_4$ have also been
deleted.
$$
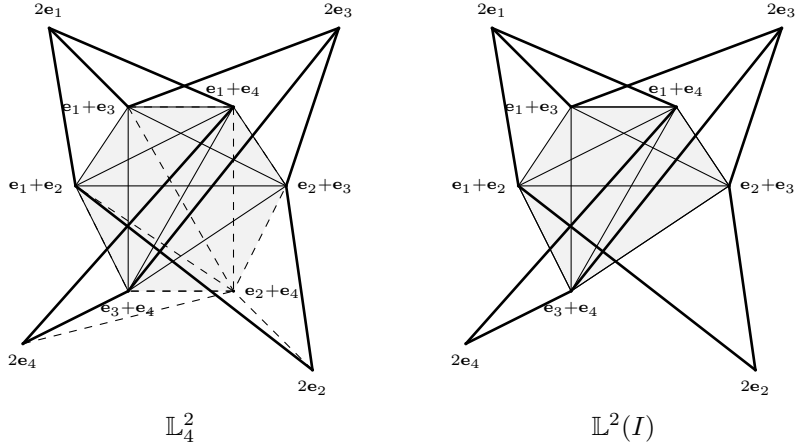
\begin{tabular}{ccc}
\begin{tikzpicture}[x=0.35cm, y=0.35cm]
\tikzstyle{point}=[inner sep=0pt]
\node (a)[point,label=left:{\tiny $\be_1{+}\be_3$}] at (5,10)  {};
\node (b)[point,label=above:{\tiny $\be_1{+}\be_4$}] at (9,10)  {};
\node (c)[point,label=right:{\tiny $\be_2{+}\be_3$}] at (11,7)   {};
\node (d)[point,label=right:{\tiny $\be_2{+}\be_4$}] at (9,3)   {};
\node (e)[point,label=below:{\tiny $\be_3{+}\be_4$}] at (5,3)   {};
\node (f)[point,label=left:{\tiny $\be_1{+}\be_2$}] at (3,7)   {};
\node (g)[point,label=above:{\tiny $2\be_1$}] at (2,13)   {};
\node (h)[point,label=above:{\tiny $2\be_3$}] at (13,13)   {};
\node (i)[point,label=below:{\tiny $2\be_2$}] at (12,0)   {};
\node (j)[point,label=below:{\tiny $2\be_4$}] at (1,1)   {};;

\draw [dashed][fill=gray!10](a.center) -- (b.center) -- (c.center) -- (d.center) -- (e.center) -- (f.center);

\draw (a.center) -- (b.center);
\draw  (a.center) -- (c.center);
\draw [dashed](a.center) -- (d.center);
\draw (a.center) -- (e.center);
\draw (a.center) -- (f.center);
\draw (b.center) -- (c.center);
\draw [dashed](b.center) -- (d.center);
\draw (b.center) -- (e.center);
\draw (b.center) -- (f.center);
\draw (c.center) -- (e.center);
\draw (c.center) -- (f.center);
\draw [dashed](d.center) -- (e.center);
\draw [dashed](d.center) -- (f.center);
\draw (e.center) -- (f.center);
\draw [line width=1pt ] (g.center) -- (a.center);
\draw [line width=1pt  ] (g.center) -- (b.center);
\draw[line width=1pt  ]  (g.center) -- (f.center);
\draw [line width=1pt  ] (h.center) -- (a.center);
\draw [line width=1pt ] (h.center) -- (c.center);
\draw [line width=1pt  ] (h.center) -- (e.center);
\draw [line width=1pt ] (i.center) -- (c.center);
\draw [dashed] (i.center) -- (d.center);
\draw [line width=1pt  ] (i.center) -- (f.center);
\draw [line width=1pt  ] (j.center) -- (b.center);
\draw [dashed  ] (j.center) -- (d.center);
\draw [line width=1pt  ] (j.center) -- (e.center);

\draw[black, fill=black] (a) circle(0.04);
\draw[black, fill=black] (b) circle(0.04);
\draw[black, fill=black] (c) circle(0.04);
\draw[black, fill=black] (d) circle(0.04);
\draw[black, fill=black] (e) circle(0.04);
\draw[black, fill=black] (f) circle(0.04);
\draw[black, fill=black] (g) circle(0.04);
\draw[black, fill=black] (h) circle(0.04);
\draw[black, fill=black] (i) circle(0.04);
\draw[black, fill=black] (j) circle(0.04);
\end{tikzpicture} & \quad \quad & 
\begin{tikzpicture}[x=0.35cm, y=0.35cm]
\tikzstyle{point}=[inner sep=0pt]
\node (a)[point,label=left:{\tiny $\be_1{+}\be_3$}] at (5,10)  {};
\node (b)[point,label=above:{\tiny $\be_1{+}\be_4$}] at (9,10)  {};
\node (c)[point,label=right:{\tiny $\be_2{+}\be_3$}] at (11,7)   {};
\node (e)[point,label=below:{\tiny $\be_3{+}\be_4$}] at (5,3)   {};
\node (f)[point,label=left:{\tiny $\be_1{+}\be_2$}] at (3,7)   {};
\node (g)[point,label=above:{\tiny $2\be_1$}] at (2,13)   {};
\node (h)[point,label=above:{\tiny $2\be_3$}] at (13,13)   {};
\node (i)[point,label=below:{\tiny $2\be_2$}] at (12,0)   {};
\node (j)[point,label=below:{\tiny $2\be_4$}] at (1,1)   {};;

\draw [fill=gray!10](a.center) -- (b.center) -- (c.center) -- (e.center) -- (f.center);

\draw (a.center) -- (b.center);
\draw (a.center) -- (c.center);
\draw (a.center) -- (e.center);
\draw (a.center) -- (f.center);
\draw (b.center) -- (c.center);
\draw (b.center) -- (e.center);
\draw (b.center) -- (f.center);
\draw (c.center) -- (e.center);
\draw (c.center) -- (f.center);
\draw (e.center) -- (f.center);
\draw [line width=1pt  ] (g.center) -- (a.center);
\draw [line width=1pt  ] (g.center) -- (b.center);
\draw[line width=1pt  ]  (g.center) -- (f.center);
\draw [line width=1pt  ] (h.center) -- (a.center);
\draw [line width=1pt  ] (h.center) -- (c.center);
\draw [line width=1pt  ] (h.center) -- (e.center);
\draw [line width=1pt  ] (i.center) -- (c.center);
\draw [line width=1pt  ] (i.center) -- (f.center);
\draw [line width=1pt  ] (j.center) -- (b.center);
\draw [line width=1pt  ] (j.center) -- (e.center);

\draw[black, fill=black] (a) circle(0.04);
\draw[black, fill=black] (b) circle(0.04);
\draw[black, fill=black] (c) circle(0.04);;
\draw[black, fill=black] (e) circle(0.04);
\draw[black, fill=black] (f) circle(0.04);
\draw[black, fill=black] (g) circle(0.04);
\draw[black, fill=black] (h) circle(0.04);
\draw[black, fill=black] (i) circle(0.04);
\draw[black, fill=black] (j) circle(0.04);
\end{tikzpicture}\\
$\LL_4^2$ & \quad \quad &  $\LL^2(I)$ \\
\end{tabular}
$$

Next, consider $I^3$. The equation $m_1m_3=m_2m_4$ produces four
non-trivial equivalence classes from \cref{d:LrI}, namely the sets
$$ \{ 2\be_1+\be_3, \be_1+\be_2+\be_4\} \quad \{ \be_1+\be_2+\be_3, 2\be_2+\be_4\}, \quad \{\be_1+2\be_3, \be_2+\be_3+\be_4\}, 
\quad \{\be_1+\be_3+\be_4, \be_2+2\be_4\}.$$ 

One can check that the above relations are the only ones. Since $r=3$, we have $s=2$ and so
all the above duplicated vertices are in $B^3$. In particular,
$\LL^3(I)$ will be the induced subcomplex of $\LL_4^3$ on a vertex set
$V$ with $16$ vertices. The sets below are two different possible sets
of vertices $V$ for $\LL^3(I)$:
$$V=\left\{ 3\be_1, 2\be_1+\be_2, 2\be_1+\be_4, 3\be_2, 2\be_2+\be_1, 2\be_2+\be_3, 3\be_3, 2\be_3+\be_2, 2\be_3+\be_4, 3\be_4, 2\be_4+\be_1, 2\be_4+\be_3, 
\right.$$
$$\left. \be_1+\be_2+\be_3, \be_1+\be_2+\be_4, \be_1+\be_3+\be_4, \be_2+\be_3+\be_4 \right\}$$
or
$$V=\left\{ 3\be_1, 2\be_1+\be_2, 2\be_1 + \be_3, 2\be_1+\be_4, 3\be_2, 2\be_2+\be_1, 2\be_2+\be_3, 2\be_2+\be_4, 3\be_3, 2\be_3+\be_1, 2\be_3+\be_2, 2\be_3+\be_4, 
\right.$$
$$\left. 3\be_4, 2\be_4+\be_1, 2\be_4+\be_2, 2\be_4+\be_3 \right\}$$
 \end{example}

\begin{example}\label{eg:9vars}
Let $S=\sfk[a,b,c,d,e,f,x,y,z]$ and $$I=(xyz, abc, def, xza, xzb, xyc, xyd, yze, yzf) = (m_1,m_2,
\ldots, m_9).$$ We have  the following relation: 
$${m_1}^4m_2m_3 =
m_4m_5m_6m_7m_8m_9\,.$$
For $r=6$ we have
$s=3$, and so $$4\be_1+\be_2+\be_3 \in F_1^6 \setminus B^6 \qbut \be_4+\be_5+\be_6+\be_7+\be_8+\be_9\in B^6.$$
Therefore, the induced subcomplex $\mathbb L^6(I)$ would not contain the vertex
$4\be_1+\be_2+\be_3$ corresponding to ${m_1}^4m_2m_3$.
\end{example}

While the examples above show that, in general, $\Lr(I)$ is smaller than $\Lrq$, there are also cases when the two complexes are the same.  In \cref{s:min} we identify an ideal $\E_q$ with $\Lr(\E_q)=\Lrq$ for each $q$. 
The two complexes are also equal for all $I$ with $q\le 3$, as shown below.

\begin{proposition} \label{p:Lr=L(I)}
  Let $I$ be an ideal minimally generated by
  square-free monomials $m_1,\ldots,m_q$ in the polynomial ring
  $S$. If $q \leq 3$ then $\Lr(I)=\Lrq$ for all $r \geq 1$.
\end{proposition}

\begin{proof} If $q=1$, then $I^r=({m_1}^r)$ for all $r$, and both
  $\Lr(I)$ and $\Lrq$ consist of a single point, so the equality holds.

 Assume $1< q \le 3$. By \cref{d:LrI} it suffices to show that
 $\Lr(I)$ and $\Lrq$ have the same vertex set, or in other
 words
   $$\bma \mid \bmb \implies  \ba=\bb \qforall \ba,\bb\in \Nrq.$$
 
  Suppose $q=2$ and $\bma \mid \bmb$ with
 $$\ba = (a_1,a_2), \quad \bb=(b_1,b_2), \qand a_1+a_2=b_1+b_2=r.$$ If
  $\ba\ne \bb$, then we may assume $a_1>b_1$ and $a_2<b_2$.  Since the
  monomials $m_i$ are square-free, we have
 $${m_1}^{a_1}{m_2}^{a_2} \mid {m_1}^{b_1}{m_2}^{b_2}\implies {m_1}^{a_1-b_1}
 \mid {m_2}^{b_2-a_2}\implies m_1\mid m_2\,,$$
which  contradicts the
 minimality of the generators of $I$.
 
Suppose $q=3$ and $\bma \mid \bmb$ with $$\ba
=(a_1,a_2,a_3),\quad \bb = (b_1,b_2,b_3) \qand
a_1+a_2+a_3=b_1+b_2+b_3=r.$$ 

Assume $\ba\ne \bb$.  If $a_i=b_i$ for some $i$ then we can reduce to
the case $q=2$, so we may assume $a_i\ne b_i$ for all $i$. Without
loss of generality, assume $a_1>b_1$.

We have three cases:

  \begin{enumerate}
 \item Suppose $a_2 >b_2$. In this case, we must also have
   $a_3<b_3$. Then $$ {m_1}^{a_1}{m_2}^{a_2}{m_3}^{a_3} \mid
   {m_1}^{b_1}{m_2}^{b_2}{m_3}^{b_3}\implies {m_1}^{a_1-b_1}{m_2}^{a_2-b_2} \mid
   {m_3}^{b_3-a_3}\implies {m_1}^{a_1-b_1} \mid {m_3}^{b_3-a_3}.$$ Since
   $m_1$ and $m_3$ are square-free, this implies that $m_1\mid m_3$
   which gives a
   contradiction. 

\item Suppose $a_2 < b_2$ and
  $a_3 < b_3$.  We have  
   $${m_1}^{a_1}{m_2}^{a_2}{m_3}^{a_3} \mid
  {m_1}^{b_1}{m_2}^{b_2}{m_3}^{b_3}\implies {m_1}^{a_1-b_1} \mid
  {m_2}^{b_2-a_2}{m_3}^{b_3-a_3}.$$ Let $x$ be a variable such that $x
  \mid m_1$. It follows that $x\mid m_2$ or $x\mid m_3$. Assume
  $x\nmid m_3$.  We then have $x^{a_1-b_1}\mid {m_2}^{b_2-a_2}$. Since
  $m_2$ is square-free, it follows that $b_2-a_2\ge
  a_1-b_1$. This is a contradiction, since
$$a_1-b_1=b_2-a_2+b_3-a_3>b_2-a_2.$$

 A similar contradiction is obtained when $x\nmid m_2$. We conclude
 that for every $x \mid m_1$, we must have $x\mid m_2$ and $x\mid
 m_3$. This implies $m_1 \mid m_2$ and $m_1\mid m_3$,
 which is a contradiction.
 
\item Suppose $a_2<b_2$ and $a_3>b_3$. After relabeling, this case reduces
  to Case (1).
\end{enumerate}
\end{proof}

\begin{remark} Note that the square-free assumption is necessary
  in \cref{p:Lr=L(I)}, for if $I = (x^2y, yz^2, xyz)$, then $m_1m_2
  \mid {m_3}^2$ in $I^2$, but $(1, 1, 0) \neq (0,0,2)$ in
  $\N_3^2$.
\end{remark}

The next proposition shows that the vertices labeled by $r\be_i$ belong
to the induced subcomplex $\Lr(I)$ for all $i\in [q]$, regardless of
the choices made in \cref{d:LrI}. Moreover, if $q\geq 2$, then for
each $i\in [q]$ there exists some $j\in[q]\setminus \{i\}$ such that
the vertex labeled by $(r-1)\be_i + \be_j$ belongs to $\Lr(I)$.

In what follows we use the standard notation $\LCM(I^r)$ to denote the
{\bf lcm lattice} of $I^r$, which is the set of all least common
multiples of subsets of the minimal monomial generating set of $I^r$,
partially ordered by division.

\begin{proposition}\label{p:irredundant-gens-Ir}
Let $r \geq 1$, $I$ an ideal in $S$ minimally generated by square-free
monomials $m_1,\ldots,m_q$, and $i\in[q]$.
\begin{enumerate}
\item[(i)] If $\bma \mid {m_i}^r$, then $\ba=r\be_i$, for any $\ba \in \Nrq$.
\item[(ii)] If $q \geq 2$ and ${m_i}^{r-1} \mid M$ for some monomial $M \neq {m_i}^r$ with $M \in \LCM(I^r)$, then there exists $j\in [q]\setminus \{i\}$ such that $m_j
  \mid M$ and, for every $\ba \in \Nrq$,
  $$\bma \mid {m_i}^{r-1}m_j \iff
    \bma =    {m_i}^{r-1}m_j    \iff
    \ba=(r-1)\be_i+\be_j.
  $$
\end{enumerate}

In particular, for all $i$, and all $j$ as in~$(ii)$,
$$r\be_i \qand (r-1)\be_i+\be_j$$ are vertices of $\Lr(I)$, and
$${m_i}^r \qand {m_i}^{r-1}m_j$$ are minimal monomial
generators of $I^r$.
\end{proposition}

  \begin{proof} We first observe that, for $g,h>0$ and $u,v\in [q]$, we have:
    \begin{equation}\label{e:fact}
                {m_u}^{h} \mid {m_v}^{g}\quad   \Longrightarrow  \quad u=v 
    \end{equation}
Indeed, since $m_v$ is square-free, if $m_u^{h}$ divides $m_v^g$, then  $m_u\mid m_v$, and we conclude $u=v$ because $m_u, m_v$ are minimal generators. 
    
  Let $\ba=(a_1,\ldots,a_q) \in \Nrq$, and suppose $\bma \mid
  {m_i}^r$. If, for some $j \in [q]$, $a_j \neq 0$ then ${m_j}^{a_j}
  \mid {m_i}^r$, which by \eqref{e:fact} implies that $j=i$, which
  results in $\ba=r\be_i$. Thus $(i)$ holds.

  We now prove~$(ii)$. If $r=1$ the statement is trivial, so assume $r, q \geq 2$ and $M \in \LCM(I^r)$
  satisfies ${m_i}^{r-1}\mid M$ but $M \neq {m_i}^r$. Then there
  exists $k\in [q]$ with $k \neq i$ and $m_k \mid M$. Let $$A=\{k\in
  [q]\colon m_k\mid M\}.$$ Choose $j \in A \smallsetminus\{i\}$ such that $m_j$
  has the fewest number of variables not dividing $m_i$. 

  Now suppose for some $\ba=(a_1,\ldots,a_q) \in \Nrq$,
  \begin{equation}\label{e:bma}
    \bma={m_1}^{a_1}\cdots {m_q}^{a_q} \mid {m_i}^{r-1}m_j
  \end{equation}

  If $a_j \geq 2$, by canceling a copy of $m_j$ in \eqref{e:bma} we
  will have $m_j^{a_j-1} \mid {m_i}^{r-1}$ which by \eqref{e:fact}
  implies that $i=j$, a contradiction.  Similarly if $a_i = r$ then
  by canceling a copy of ${m_i}^{r-1}$ in \eqref{e:bma} we obtain
  $m_i \mid m_j$, again a contradiction since $m_i$ and $m_j$ are minimal generators.
  So we must have $a_j \leq 1$ and $a_i \leq r-1$.

  Now we claim that we must have $a_j=1$ or $a_i=r-1$. Otherwise, if
  $a_j=0$ and $a_i \leq r-2$, since $a_1+\cdots + a_q=r$, there must exist $c,d\in
  [q]\smallsetminus\{i,j\}$ with $a_c,a_d>0$ ($c$ and $d$ could be equal). In particular by \eqref{e:bma}
\begin{equation}\label{e:cd}
m_cm_d \mid {m_i}^{r-1}m_j.
\end{equation}
Let 
\begin{equation}\label{e:cj}
  m_c=\gcd(m_c,m_i) \cdot m_c', \quad
  m_d=\gcd(m_d,m_i) \cdot m_d' \qand
  m_j=\gcd(m_j,m_i) \cdot m_j'.
\end{equation}

From \eqref{e:cj} one can see that $m'_c$ and $m'_d$ do not share any variables with $m_i$, and so by \eqref{e:cd}
\begin{equation}\label{e:primes-div}
  m_c'm_d' \mid m_j' 
\end{equation} and also
\begin{equation}\label{e:primes}
  m_c' \mid m_j' \mid m_j \mid M \qand
  m_d' \mid m_j' \mid m_j \mid M.
\end{equation}
On the other hand $m_i \mid M$, so \eqref{e:cj} and \eqref{e:primes},
together with the fact that $m_c$ and $m_d$ are square-free, imply
that $$m_c \mid M \qand m_d \mid M,$$ which in turn implies that $c,d
\in A$. Assume without loss of generality that $\deg m_c' \leq \deg
m_d'$. The fact that $j \in A$ was picked so that $m_j$ has the fewest
number of variables outside $m_i$ and \eqref{e:primes-div} together
imply
$$  \deg m_c'+\deg m_d ' \leq
\deg m_j' \leq \deg m_c' \leq \deg m_d'; $$
a contradiction since $m'_c$, $m'_d$ and $m'_j$ all have positive degrees.

So the only possibilities are $a_j=1$ or $a_i=r-1$.  If $a_j=1$
then by~$(i)$ we have $\bma ={m_i}^{r-1}m_j$, and we are done.  If
$a_i=r-1$ we have $\bma ={m_i}^{r-1}m_k \mid {m_i}^{r-1}m_j$ for some
$k \in [q]$, hence by cancellation $k=j$ and 
$\bma={m_i}^{r-1}{m_j}$.

The remaining statements now follow immediately from $(i)$, $(ii)$ and
\cref{d:LrI}.
\end{proof}

One additional lemma is necessary before the
statement of our main result.

\begin{lemma}\label{l:non-div}
Using \cref{n:rqs}, let $i,j \in [q]$ with $i
\not = j$, and let $\ba=(a_1, \ldots, a_q)$ and $\bb=(b_1, \ldots,
b_q)$ be in $\Nrq$. If $\alpha$ is a non-negative integer such that
$$a_i \ge \alpha\qand b_j> r-\alpha
$$
Then
\begin{enumerate}
\item[(i)] $\bmb \nmid \bma$;
  \smallskip

\item[(ii)] 
${m_i}^{\alpha}{m_j}^{a_i-\alpha}\frac{\bma}{{m_i}^{a_i}}\mid \lcm(\bma, \bmb)$;

\item[(iii)] $\deg({m_i}^{\alpha}{m_j}^{a_i-\alpha}\frac{\bma}{{m_i}^{a_i}}) \leq
  \deg(\bma) \iff \deg(m_i) \geq \deg (m_j)$ or $a_i=\alpha$.

\end{enumerate}
\end{lemma}

\begin{proof} To prove $(i)$, 
assume $\bmb\mid \bma$. Set
$$
\ba'=\ba-a_i\be_i\in \N^{r-a_i}_q
$$
so that $$\bma=\m^{\ba'}{m_i}^{a_i}\,. $$

Let $x$ be a variable such that  $x\mid m_j$. Then $x^{b_j} \mid \bmb$ and hence $x^{b_j}\mid \bma$.  Suppose $x$ does not divide $m_i$. Then $x^{b_j}\mid\m^{\ba'}$. Since $\m^{\ba'}$ is a product of $r-a_i$ square-free monomials, we have then 
$$
b_j\le r-a_i\,.
$$
The hypothesis on $a_i$ and $b_j$ gives 
$$r-a_i\le r-\alpha< b_j\,.$$
The  contradiction that arises shows $x\mid m_i$. 
 We conclude $m_j\mid m_i$, hence $i=j$, a contradiction. 

 Now we prove~$(ii)$. 
Set 
$$ \ba'=\ba-a_i\be_i\in \N^{r-a_i}_q \qand \bb'=\bb-b_j\be_j\in
\N^{r-b_j}_q
$$
so that $$\bma=\m^{\ba'}{m_i}^{a_i}\qand \bmb=\m^{\bb'}{m_j}^{b_j}\,.$$
  It
  will be shown that 
$$\deg_x({m_i}^{\alpha}{m_j}^{a_i-\alpha}\m^{\ba'}) \leq \max\{ \deg_x(\bma), \,
  \deg_x(\bmb)\}$$ for all variables $x$, where for a monomial $M$, $\deg_x(M) = \max\{t
  \st x^t \mid M\}.$
    
Let $u = \deg_x(\m^{\ba'})$ and  $v=\deg_x(\m^{\bb'})$.  Since  $\m^{\ba'}$ is a product of $r-a_i$ square-free monomials, and $\m^{\bb'}$ is a product of $r-b_j$ square-free monomials, we have 
\begin{equation}\label{e:uv-ab}
  0 \le u \le r-a_i\qand 0 \le v \le r-b_j.
\end{equation}
Also note that
  $$\deg_x({m_i}^{\alpha}{m_j}^{a_i-\alpha}\m^{\ba'}) = \alpha \cdot \deg_x(m_i)+ (a_i-\alpha)\cdot \deg_x(m_j)+ u.$$

\begin{itemize}
  \item If $x \mid m_i$, then  
    \begin{align*}
      \deg_x({m_i}^{\alpha}{m_j}^{a_i-\alpha}\m^{\ba'}) = & \, \alpha+  (a_i-\alpha)\deg_x(m_j) + u\\
      \leq & \, \alpha+  (a_i-\alpha) + u
      = a_i+u\\
      = & \, \deg_x(\bma) \\
      \leq & \, \max \{ \deg_x(\bma), \deg_x(\bmb) \}.
    \end{align*}

   \item If $x \nmid m_i$, but $x \mid m_j$, then 
     \begin{align*}
     \deg_x({m_i}^{\alpha}{m_j}^{a_i-\alpha}\m^{\ba'}) = & \, (a_i-\alpha) + u\\
      \leq  & \, (a_i-\alpha) + (r-a_i)  \quad \mbox{by \eqref{e:uv-ab}}\\
      = & \, r-\alpha<  b_j \le  b_j+v \\
      = & \, \deg_x(\bmb)\\
      \leq & \, \max \{ \deg_x(\bma), \deg_x(\bmb) \}.
    \end{align*}

   \item If $x\nmid m_i$ and $x\nmid m_j$, then
     $$\deg_x({m_i}^\alpha{m_j}^{a_i-\alpha}\m^{\ba'}) = \deg_x(\m^{\ba'}) = \deg_x(\bma) \leq \max \{
     \deg_x(\bma), \deg_x(\bmb) \}.$$
\end{itemize}

   This finishes the proof of~$(ii)$.

   Finally,~$(iii)$ follows directly from the fact that
\begin{align*}
\deg({m_i}^{\alpha}{m_j}^{a_i-\alpha}\m^{\ba'}) &= \alpha \cdot \deg m_i+(a_i-\alpha) \cdot \deg m_j + \deg (\m^{\ba'}) \\
&= \alpha \cdot \deg m_i+(a_i-\alpha) \cdot \deg m_j + \deg (\bma) - a_i \cdot \deg m_i \\
&= \deg(\bma)-(a_i-\alpha)(\deg(m_i)-\deg(m_j)). 
\end{align*}
\end{proof}

\begin{remark}
\label{r:connection}
Recall that $s=\left\lceil{\frac{r}{2}}\right\rceil$. When $\alpha=s$, the condition on $b_j$ in \cref{l:non-div} can be translated as follows: 
$$
b_j>r-s\iff \begin{cases} \text{$b_j\ge s$ \quad when $r$ is odd}\\
\text{$b_j>s$ \quad when $r$ is even.}
\end{cases} 
$$
\end{remark}

Recall that, by \cref{t:res-by-quasitree}, if $\D$ is a quasi-tree
with vertices labeled with the monomial generating set of a monomial
ideal $I \subseteq \sfk[x_1, \ldots, x_n]$, then $\D$ supports a
resolution of $I$ if and only if $\D_M$ is connected for every
monomial $M$. It is easy to check that the connectivity need only be
verified for $M$ in $\LCM(I)$.  We now use this criterion to prove the
main result of the paper.

\begin{theorem}[{\bf Main Result}]\label{t:main}
  Let $q\ge 1$ and let $I$ be a monomial ideal minimally generated by
  square-free monomials $m_1, \dots, m_q$.  By labeling each vertex
  $\ba$ of the simplicial complexes below with the monomial $\bma$, the
  following hold for any $r\ge 1$:
\begin{enumerate} 
\item $\Lrq$ supports a free resolution of $I^r$;
\item $\Lr(I)$ supports a free resolution of $I^r$.
\end{enumerate}
\end{theorem}

\newcommand{\Ll}{\mathbb{L}}

\begin{proof} 
In this proof, we let $\Ll$  denote either $\Lrq$ or  $\Lr(I)$.  Fix $\Ll$ as one of these two choices. Let $V$ denote the set of vertices of $\Ll$ and let
$\{m_1,\ldots,m_q\}$ be the minimal monomial generating set of $I$.
Following \cref{n:rqs}, for $\ba = (a_1, \ldots, a_q) \in
\Nrq$ we let $\bma = {m_1}^{a_1} \cdots {m_q}^{a_q}$ and $\be_1, \ldots,
\be_q$ denote the standard basis vectors for $\mathbb R^q$.

We willl show that for every $M$ in $\LCM(I^r)$, ${\Ll}_M$ is empty or connected, where ${\Ll}_M$ is the induced subcomplex of the complex $\Ll$ on
  the set 
$$V_M=\{\ba \in V \st \bma \mid M\}\,.$$ 
By \cref{p:Lr-quasitree}, $\Lrq$ is a quasi-tree and by \cref{p:qtree-induced}, $\Lr(I)$ is a quasi-forest, which is connected and thus a quasi-tree.  In view of \cref{t:res-by-quasitree}, we can then conclude that $\Ll$ supports a free resolution of $I^r$.

Suppose $M \in \LCM(I^r)$ and ${\Ll}_M\neq \emptyset$. If $M = {m_i}^r$ for some $i\in [q]$, then
${\Ll}_M$ is a point by \cref{p:irredundant-gens-Ir}~$(i)$, and
hence is connected.

\smallskip

Assume $M\ne
{m_i}^r$ for all $i\in [q]$. Note that this implies $q>1$.

The facets of ${\Ll}_M$ are the maximal sets, under inclusion, among the
sets 
\begin{equation}
\label{e:list}
F^r_1\cap V_M, \ldots, F^r_{q}\cap V_M, G^r_1\cap V_M,  \ldots, 
G^r_{q}\cap V_M.
\end{equation}
Note that not all these sets  are facets of ${\Ll}_M$, but  all
facets of ${\Ll}_M$ are among those listed in \eqref{e:list}.
We refer to the facets of ${\Ll}_M$ of form
$F^r_i \cap V_M$ as {\it the first layer},  and those of the form $G^r_i \cap V_M$ as {\it the second layer}. We refer to $B^r\cap
V_M$ as the {\it base} of ${\Ll}_M$. The base $B^r\cap V_M$
could become empty, depending on the choice of $M$.

We use the faces in \eqref{e:list} to argue the connectedness of
${\Ll}_M$ as follows: Claim~1 below shows that any facet of
${\Ll}_M$ in the second layer is connected to a facet in the first
layer. Claim~2 implies that any two facets of ${\Ll}_M$ that are in
the first layer connect through the nonempty base.  The combination of
these two facts implies that ${\Ll}_M$ is connected, which will end
our proof.

\subsection*{\bf Claim 1: (The second layer facets of ${\Ll}_M$
  intersect the first layer facets)} {\it For any $i\in [q]$ we have:}
$$G^r_{i}\cap V_M\ne \emptyset\implies G^r_{i}\cap F^r_{i}\cap
  V_{M}\ne \emptyset .$$

\subsubsection*{Proof of Claim~1} 
Assume $G^r_{i}\cap V_M\ne \emptyset$ for some $i\in [q]$.
By \cref{d:Lrq}, there exists $a\in [q]$ such that $$\bd = (r-1)\be_i +
\be_a \in G^r_{i}\cap V_M\,.$$

If $i\ne a$, then $\bd\in F^r_i$ as well, and hence
$G^r_{i}\cap F^r_{i}\cap V_M \ne \emptyset$ as desired.

Assume $a=i$. Then $\bd\in V_M$ implies ${m_i}^r\mid M$. Since $M\neq
{m_i}^r$, \cref{p:irredundant-gens-Ir} guarantees that there exists
$j\in [q]$ with $j \neq i$ and $m_j \mid M$ so that $(r-1)\be_i+\be_j
\in V$.

 Since ${m_i}^r\mid M$ and $m_j\mid M$, we see that ${m_i}^{r-1}m_j\mid
 M$.  Indeed, we set $M={m_i}^rM'$ and, since $m_j$ is square-free, we
 have $m_j\mid m_iM'$, hence ${m_i}^{r-1}m_j\mid {m_i}^rM'=M$.
Thus we have
$$(r-1)\be_i+\be_j \in G^r_{i}\cap F^r_{i}\cap V_M.$$

\subsection*{\bf Claim 2: (The first layer facets connect through the
  base of ${\Ll}_M$)} {\it For any $i,j\in [q]$ with $i\ne j$ we
  have:}
$$ \mbox{\rm{if}} \quad (F^r_i\cup G^r_{i})\cap V_M\ne \emptyset\quad\text{and}\quad
 (F^r_j\cup G^r_{j})\cap V_M\ne \emptyset \quad \mbox{\rm{then}} \quad B^r \cap V_M\ne
 \emptyset .$$

 \subsubsection*{Proof of Claim~2}
 Assume for some $i,j\in [q]$ with $i\ne j$ that 
$$ (F^r_i\cup G^r_{i})\cap V_M\ne \emptyset \qand
 (F^r_j\cup G^r_{j})\cap V_M\ne \emptyset.$$
Without loss of generality, assume $\deg(m_i)\ge \deg(m_j)$. We choose 
 $$\ba = (a_1, \ldots, a_q) \in (F^r_i\cup G^r_{i})\cap V_M\qand
 \bb = (b_1, \ldots, b_q) \in (F^r_j\cup G^r_{j})\cap V_M$$
 such that 
 \begin{equation}
 \label{e:choose-a}
 \deg(\bma)=\min\{\deg(\m^\bd)\mid \ \bd\in (F^r_i\cup G^r_{i})\cap V_M\}\,.
 \end{equation}
 Assume, by way of contradiction, that $B^r\cap V_M=\emptyset$.  
   Since $\ba, \bb \in V_M$, we have then $\ba, \bb \notin B^r$.
   Therefore, $$s <a_i,b_j \le r \qand \bma, \bmb \mid M.$$

 Set $$\ba'=\ba-a_i\be_i\in \N^{r-a_i}_q \qand \bc := (c_1, \ldots, c_q) =s \be_i + (a_i-s) \be_j + \sum_{1\le k\le
   q, k\ne i} a_k\be _k\in \Nrq.$$
 
 Since $b_j > s \geq r-s$, by \cref{l:non-div}~$(ii)$
 and \cref{r:connection},
 \begin{equation}
 \label{e:div2}
 \m^{\bc} ={m_i}^{s}{m_j}^{a_i-s}\m^{\ba'} \mid \lcm(\bma, \bmb) \mid M.
 \end{equation}
 
  Moreover $\bc\in B^r$, because
\begin{align*}
c_i&=s &\\
c_j&= a_i-s+a_j\le a_i-s+(r-a_i)=r-s \le s &\mbox{since } a_i+a_j \leq r\\
c_k&=a_k\leq r - a_i<r-s\le s &\mbox{for all } k \neq i,j.
\end{align*}
 
If $\Ll=\Lrq$, then $V=\Nrq$. Since \eqref{e:div2} shows $\bc\in V_M$, we conclude $\bc\in B^r\cap V_{M}$, a contradiction. This finishes the proof  of part (1). 

Assume now $\Ll=\Lr(I)$. Since we assumed $B^r\cap V_{M}=\emptyset$, we must have  $\bc \notin V$. \cref{d:LrI} implies then $$\m^{\bc'} \mid \m^{\bc}
\qforsome \bc':=(c_1', \ldots, c_q') \in V.$$

Note that $\m^{\bc'}\mid M$ by \eqref{e:div2} and hence $\bc'\in
V_M$. Since $B^r\cap V_{M}=\emptyset$, we must have $\bc'\notin
B^r$. Using the notation in \cref{d:LrI}, let $V_a$ be the equivalence
class of which $\bc'$ is a representative. If $\m^{\bc'}=\m^{\bc}$,
then $\bc\in V_a$ as well, and hence $V_a\cap B^r\neq\emptyset$. Then
Step 1 in \cref{d:LrI} implies $\bc'\in B^r$, a contradiction. Hence
$\m^{\bc'}\ne \m^{\bc}$.
 
 Since  $\bc'\notin B^r$, there exists $k\in [q]$ such that
 $c'_k>s\ge r-s$. Since $c_i=s$ and $\m^{\bc'}\mid \m^{\bc}$, 
 \cref{l:non-div}~$(i)$ implies $k=i$.  In particular,
 $c'_i>s$, and thus  
 $$
   \bc'\in (F^r_i\cup G^r_i)\cap V_M\,. 
 $$
Since $\m^{\bc'}$ strictly divides
$\m^{\bc}={m_i}^{s}{m_j}^{a_i-s}\m^{\ba'}$, using \cref{l:non-div}~$(iii)$, we have 
 $$\deg(\m^{\bc'})<\deg(\m^\bc) \leq
\deg(\bma),$$ contradicting the assumption made in
\eqref{e:choose-a}. This finishes the proof of part (2). 
\end{proof}

\begin{remark}
	One may feel that part $(1)$ of the statement above is weaker than $(2)$, since $\LL^r(I)\subseteq \Lrq$. However, the remarkable aspect of $(1)$ is that, before labeling, $\Lrq$ does not depend on the ideal $I$. Thus the same topological structure, depending only on $r$ and $q$, supports a free resolution of the $r^{th}$ power of {\bf any} ideal generated by $q$ square-free monomials. The idea is that $\Lrq$ provides an alternative notion to the Taylor complex for powers of $I$: when $r=1$, $\Lrq$ is the Taylor simplex, and when $r>1$, $\Lrq$ is significantly smaller than the Taylor simplex but still supports a resolution of $I^r$. 
\end{remark}

%%%%%%%%%%%%%%%%%%%%%%%%%%%%%

\section{Bounds on Betti numbers}\label{s:bounds}

%%%%%%%%%%%%%%%%%%%%%%%%%%%%%
 
 One of the key applications of \cref{t:main} is that we are able to improve the bounds on Betti numbers for powers of ideals from that given by the standard Taylor resolution of $I^r$ since $$\Lr(I)
\subseteq \Lrq \subseteq \Taylor(I^r).$$ This section contains computations that illustrate the extent to which our results improve the Taylor bounds. 

When $r=2$, we were able to provide a concrete formula for the number
of faces of $\Lr(I)$ in~\cite{L2}, and as a result we provided
 bounds for Betti numbers of $I^2$, which are shown in \cref{p:last} to be sharp. 
 When $r>2$, however, such
formulas are not as easy to write and the numbers are very large even
for small examples. In this case we shift our attention to $\Lrq$. While the
bound on the Betti numbers stemming from $\Lr(I)$ is dependent on the
relations among the generators of $I$, one can use the structure of
$\Lrq$ to get a general, albeit weaker, bound on $\beta_t(I^r)$.

\begin{theorem}[{\bf Bounds on Betti numbers of $I^r$}]\label{t:counting}
  If $I$ is a square-free monomial ideal with $q$ minimal generators,
  then the Betti numbers of $I^r$ for $r \geq 2$ are bounded above by
$$\beta_t(I^r) \leq q\left( {{q-1}\choose{t}} + {{f}\choose{t+1}} - {{b}\choose{t+1}}\right) + {{b}\choose{t+1}}$$
  where $t \geq 0$, $b$ is the coefficient of $x^r$ in the expansion of 
$$(1 + x + x^2 + \cdots + x^{\left\lceil{\frac{r}{2}}\right\rceil})^q,$$ 
         and
$$f = \frac{{{q+r-1}\choose{r}} -b-q}{q} + b.$$
In particular,
$\pd (I^r) \leq \max \{ q-1, f-1\}.$
\end{theorem}

\begin{proof} Let $s=\left\lceil{\frac{r}{2}}\right\rceil$.
Note that $\beta_t(I^r)$ is bounded above by the number of faces of
dimension $t$ of $\Lrq$. To count the faces of dimension $t$, we use the
sets $B^r, F_i^r, $ and $G_i^r$ from \cref{d:Lrq}.

Observe that  the coefficient $b$ of $x^r$ in the expansion of
$$(1 + x + x^2 + \cdots + x^s)^q$$ is exactly the number of
$q$-tuples $(a_1,\ldots,a_q)$ with $a_i\leq s$ and $a_1+\cdots+a_q=r$,
in other words $b=|B^r|$. 

Note that the vertices of $\Lrq$ are formed by selecting $r$ of the original $q$ generators, so using the standard formula for counting with repetition, we have
$$|V(\Lrq)|= {{q+r-1}\choose{r}}.$$
Now $|F_i^r| = |F_j^r|$ for all $i,j$, so let $f=|F_i^r|$. Since
$|G_i^r \setminus F_i^r| =1$, there are $|V(\Lrq)| -q$ vertices in
$\cup _{i=1}^q F_i^r$.  Also $F_i^r \cap F_j^r = B^r$ for all $i,j$ such that $i \not = j$,
so we have

\begin{equation}
\label{e:f}
f=|F_i^r|= {\frac{1}{q}} \big( ( |V(\Lrq)| -q ) - |B^r|\big) +|B^r| = \frac{{{q+r-1}\choose{r}} -q -b}{q} + b.
\end{equation}

To count the number of faces of dimension $t$, that is, faces with
$t+1$ vertices, in $\Lrq$, we separate the faces into three distinct
types.

\begin{enumerate}
\item Faces containing ${m_i}^r$ for some $i$: By the definition of
  $\Lrq$, faces of this type must be contained in $G_i^r$. Since the
  vertex corresponding to ${m_i}^r$ has been fixed, $t$ additional
  vertices of $G_i^r$ are needed. There are $q-1$ such vertices since
  all vertices of $G_i^r$ have the form $(r-1)\be_i + \be_j$ where $1
  \leq j \leq q$. Since there are $q$ choices for $i$, there are
$$q {{q-1}\choose{t}}$$
such faces.

\item Faces contained in $B^r$: Setting $b=|B^r|$ as above, there are 
$${{b}\choose{t+1}}$$
such faces.

\item Faces contained in $F_i^r$ but not in $B^r$:
Recalling that  $f=|F_i^r|$  and $B^r \subseteq F_i^r$, there are
$$q\left( {{f}\choose{t+1}} - {{b}\choose{t+1}} \right)$$
such faces.
\end{enumerate}

Combining the three cases, we see that there are
$$q\left( {{q-1}\choose{t}} + {{f}\choose{t+1}} -
{{b}\choose{t+1}}\right) + {{b}\choose{t+1}}$$ faces of $\Lrq$ of
dimension $t$. Thus for $I^r$,
$$\beta_t(I^r) \leq q\left( {{q-1}\choose{t}} + {{f}\choose{t+1}} - {{b}\choose{t+1}}\right) + {{b}\choose{t+1}}.$$

In particular, if $t>
q-1$ and  $t+1 > f$, we must have $\beta_t(I^r)=0$. Thus
\begin{equation}\label{e:qf}
\pd (I^r) \leq \max \{ q-1, f-1 \}. %\qedhere
\end{equation}
\end{proof}

\begin{corollary}\label{c:smallr} If an ideal $I$ is minimally
  generated by $q$ square-free monomials, then the Betti numbers of $I^r$ are bounded by
        $$\beta_t (I^r) \leq q {{q-1}\choose{t}} + {{b}\choose{t+1}}$$
        for $r=2,3$, where $b$ is as defined in \cref{t:counting}.  In the case
        where $r =2$, $b = {q\choose 2}$ and the bound reduces to the bound given in \cite{L2}.
\end{corollary}

\begin{proof}
When $r=2$ or $r=3$, then $F_i^r=B^r$ for all $i$, so $b=f$ and the result follows immediately from simplifying the formula in \cref{t:counting}. Moreover, when $r=2$, the coefficient of $x^2$ in the binomial expansion of $(1+x)^q$ is ${q \choose 2}$ as stated. The reduction is then evident.
\end{proof}

Notice that when $r=2$, the bound is known to be sharp; it agrees with
the result in \cite{L2}.  In \cref{p:last} we characterize the
values of $r$ and $q$ for which these bounds are sharp, i.e. can be
realized by some ideal.

\begin{example} As a first example, we examine cases where $r$ or $q$ is small.
We first consider $b$ and $f$ for small values of $q$.

\begin{itemize}
\item For $q=2$, computing $(1+x+ x^2 + \cdots + x^s)^2$ shows that
  $b=1$ when $r=2s$ is even and $b=2$ when $r=2s-1$ is odd.  If $r$ is
  even, we then have

  \begin{equation}\label{e:first}
    f = \frac{{{2+r-1}\choose{r}} -2 -1}{2} +1
      = \frac{r+1-2-1}{2} +1
      = \frac{r}{2}=s.
  \end{equation}
  
A similar computation shows that when $r$ is odd, $f=s$ as well.
\item If $q=3$ and $r=3$ or $r=4$ (so that $s=2$), then computing $(1+x+x^2)^3$ yields $b = 7$ and $f=7$ when $r=3$ and $b=6$ and $f=8$ when $r=4$.
\end{itemize}

Applying the equations from \cref{t:counting} yields:
\begin{itemize}
\item For $r=2$, any $q$, and any $t$,
$$\beta_t(I^r) \leq  q{{q-1}\choose{t}} + {{\frac{1}{2} q(q-1)}\choose{t+1}}.$$

\item For any $r$, $q=2$, and any $t \geq 2$, if $s = \lceil{ r/2\rceil}$,
$$\beta_t(I^r) \leq 2{{s}\choose{t+1}}. $$

\item For $r=q=3$ and any $t$, 
$$\beta_t(I^r) \leq 3 {{2}\choose{t}} + {{7}\choose{t+1}}.$$

\item For $r=4$, $q=3$ and any $t \geq 3$, 
$$\beta_t(I^r) \leq 3{{8}\choose{t+1}} -2{{6}\choose{t+1}}.$$
\end{itemize}

\end{example}

\begin{remark}
\label{r:f>q}
In \eqref{e:qf} it is useful to understand which of the integers
$q-1$, $f-1$ achieves the maximum.  For small
values of $q$ and $r$, it is possible to have $q \ge f$.
However, the
opposite holds in most cases. More precisely, we show below that if $r$,
$q$ satisfy any of the following assumptions:
\begin{enumerate}
\item $q=2$ and $r\ge 5$;
\item $q=3$ and $r\ge 3$;
\item $q\ge 4$ and $r\ge 2$, 
\end{enumerate}
then $f>q$, and hence $$\pd(I^r)\le f-1=\dim \Lrq\,. $$

In the case $q=2$, \eqref{e:first} shows $f=s=\lceil{\frac{r}{2}}\rceil$. If $r\ge 5$, we must
have $s >2$, thus $f=s>2=q$. This
settles~(1).

Now suppose $q>2$. If $r=2$, then \cref{c:smallr} shows that $f=b=
{q\choose 2}$. When $q\ge 4$, we have $f={q\choose 2}>q$.  Thus, to
show~(2) and~(3) it remains to consider the case when $q\ge 3$ and
$r\ge 3$. Observe that when $r\geq 3$, $s\geq 2$.  In this scenario,
since  $b$ is the coefficient of  $x^r$ in $(1+x+\dots +x^s)^q$ and $q>2$,  then   $b\geq 2$ and  
 $bq-b> q$. 

To see that  $f>q$ holds, we will show $f-q>0$, which by \eqref{e:f} amounts to
$$\frac{\binom{q+r-1}{r}-q- b}{q} +b -q 
= \frac{\binom{q+r-1}{r}-q-b+bq-q^2}{q}>0\,.$$
It is sufficient to show the numerator is positive, so we observe the following, where the first inequality results from $bq-b > q$ and the second follows since $r \geq 3$:
\begin{align*}
\binom{r+q-1}{q-1}-q-b+bq-q^2 
   &> \binom{r+q-1}{q-1}-q^2 \\
   &\geq \binom{q+2}{q-1} - q^2  \\
  &= \frac{(q+2)(q+1)q}{6}-q^2=\frac{q(q-1)(q-2)}{6}>0\,. & 
\end{align*}
This ends our argument.

\end{remark}

\begin{example}
In general, the bounds from \cref{t:counting} will be considerably
smaller than those provided by the Taylor complex, which is a simplex,
and the bound on the projective dimension will also decrease
significantly.  We display this phenomenon in the table below.

\bigskip

\renewcommand{\arraystretch}{1}
\begin{center}
\begin{tabular}{ |p{1.2cm}|p{1.7cm}|p{1.7cm}|p{1.7cm}|p{1.7cm}|p{1.7cm}|p{1.7cm}|}
 \hline
 \multicolumn{7}{|c|}{Bound Comparisons} \\
 \hline
 & \multicolumn{2}{|c|}{$q=2$, $r=3$} & \multicolumn{2}{|c|}{$q=3$, $r=3$} &
 \multicolumn{2}{|c|}{$q=3$, $r=4$} \\
 \hline
 & \cref{t:counting}&$3$-simplex&\cref{t:counting}&$9$-simplex&\cref{t:counting}&$14$-simplex\\
 \hline
 $\beta_0 (I^r)\leq $& 4&4&10&10&15&15\\
 $\beta_1 (I^r)\leq$ & 3&6&27&45&60&105\\  
 $\beta_2 (I^r)\leq$ & 0&4&38&120&131&455\\
 $\pd (I^r)\leq$     & 1&3&6&9&7&14\\
 \hline
\end{tabular}
\end{center}
\bigskip
\end{example}

The bounds on Betti numbers and projective dimension given by the complex $\Lrq$ in \cref{t:counting} are most effective when the generating set $\{ \bma \mid \ba \in \Nrq\}$ does not contain redundancies. When there are redundancies, using $\Lr(I)$ will yield improved bounds. We illustrate how to use this improvement by continuing \cref{e:square}.

  \begin{example} \label{e:square2}
Let $I=(xy,yz,zw,xw)=(m_1,m_2,m_3,m_4)$ as in \cref{e:square}.
Counting faces of size $i$ in the complex $\LL^2(I)$ provides a bound
on the $i^{th}$ Betti number of $I^2$.  Note that in general, these
improvements in the Betti numbers follow from knowledge of the
redundancies in $\{ \bma \mid \ba \in \Nrq\}$ and can often be
computed from that information using the equivalence classes without
explicitly constructing $\LL^r(I)$.  For instance, a comparison of the
bounds is summarized in the table below for $I^2$ and for $I^3$ using
the first of the two vertex sets for $\LL^3(I)$ given in
\cref{e:square}.

\renewcommand{\arraystretch}{1}
\begin{center}
\begin{tabular}{ |p{1.2cm}|p{0.7cm}|p{1.55cm}|p{1.55cm}|p{1.65cm}|p{0.7cm}|p{1.65cm}|p{1.55cm}|p{1.65cm}| }
 \hline
 \multicolumn{9}{|c|}{Bound Comparisons} \\
 \hline
 & \multicolumn{4}{|c|}{$r=2$} & \multicolumn{4}{|c|}{$r=3$}  \\
 \hline
 & $\LL^2(I)$ &$8$-simplex& $\quad \LL^2_4$  & $9$-simplex &  $\LL^3(I)$&$15$-simplex &$\quad \LL^3_4$ &$19$-simplex\\
&&&(Thm. 6.1)&&&&(Thm. 6.1)&\\
 \hline
$\beta_0 (I^r)\leq $& 9&9&10&10&16&16&20&20\\
 $\beta_1 (I^r)\leq$ & 20&36& 27 & 45 & 74&120 &132&190\\  
 $\beta_2 (I^r)\leq$ & 18&84&32&120&224&560&572&1,140\\
 $\pd (I^r)\leq$     & 4&8&5&9&11&15&15&19\\
 \hline
\end{tabular}
\end{center}

\end{example} 

%%%%%%%%%%%%%%%%%%%%%%%%%%%%%%%%%%%%%%%%%%%%%%%%%%%%%%%%%%%%%%%%%
\section{Extremal Ideals: When does $\Lrq$ support a minimal resolution?}
\label{s:min}
%%%%%%%%%%%%%%%%%%%%%%%%%%%%%%%%%%%%%%%%%%%%%%%%%%%%%%%%%%%%%%%%%
Based upon the above work, a natural question that arises is the
following: for given $r$ and $q$, can one find ideals $I$ for which
$I^r$ has a {\it minimal} free resolution supported on $\Lrq$ itself?
When $r=1$, $\Lrq$ is the Taylor complex, which one can easily see
always supports a minimal free resolution of the ideal generated by
$q$ variables $(x_1,x_2,\ldots,x_q)$. In the case where $r>1$,
\cref{p:last} describes exactly when the bounds for Betti numbers in
\cref{t:counting} and \cref{c:smallr} are sharp, in the sense that
there exist ideals for which equality is attained. We call the ideals
which realize these bounds $q$-{\it extremal ideals} and denote them by
$\E_q$. In fact we can show a much stronger statement in
\cref{t:extremal-betti}: the powers $\Erq$ have maximal Betti numbers
among the ideals $I^r$ where $I$ is generated by $q$ square-free
monomials.

\begin{definition}[{\bf Extremal ideals}]\label{d:extremal}
Let $q$ be a positive integer. For every set $A$ with $\emptyset \neq
A \subsetneq [q]$, introduce a variable $x_A$, and consider the
polynomial ring
$$S_{\E}=\sfk\big [ x_A \st \emptyset \neq A \subseteq [q]\big ]$$
over a field $\sfk$.  For each $i \in [q]$ define a square-free
monomial $\e_i$ in $S_{\E}$ as
  $$\e_i= \prod_{\substack{ A \subseteq [q]\\ i \in A}} x_A.$$ The square-free monomial ideal $\E_q =
  (\e_1,\ldots, \e_q)$ is called a {\bf $\pmb{q}$-extremal ideal}.
  \end{definition}

When it is unlikely to lead to confusion, we will simplify the
notation by writing $x_1$ for $x_{\{1\}}$, $x_{12}$ for $x_{\{1,2\}}$,
etc., and refer to a $q$-extremal ideal simply as an extremal ideal. The ring $S_{\E}$ has $2^{q} - 1$ variables, corresponding to the
power set of $[q]$ (excluding $\emptyset$), and each $\e_i$ is the
product of $2^{q-1}$ variables; that is, those corresponding to the
subsets of $[q]$ that contain $i$.  The following example illustrates how this works for $\E_4$.

\begin{example} When $q=4$, the ideal $\E_4$ is generated by the  monomials
  $$\begin{array}{ll} \e_1&=x_1 x_{12} x_{13} x_{14} x_{123} x_{124}
    x_{134}  x_{1234}  ; \\ \e_2&=x_{2} x_{12} x_{23}
    x_{24} x_{123} x_{124} x_{234} x_{1234}
    ;\\ \e_3&=x_{3} x_{13} x_{23} x_{34} x_{123} x_{134} x_{234}
    x_{1234} ;\\ \e_4&=x_{4} x_{14} x_{24} x_{34}
    x_{124} x_{134} x_{234} x_{1234} \\
  \end{array}$$
in the polynomial ring $\sfk[x_1, x_2, x_3, x_4, x_{12}, x_{13}, x_{14}, x_{23}, x_{24}, x_{34}, x_{123}, x_{124}, x_{134}, x_{234}, x_{1234}]$. 
\end{example}

Using the terminology of \cite{PV}, it naturally follows that $\E_q$
is precisely the {\it nearly Scarf ideal} of a $(q-1)$-simplex: we see
this by matching the variable $x_{[q]\smallsetminus B}$ with the face
$B$ of the simplex.

Following \cref{n:rqs}, let $r$ and $q$ be positive integers and $I$
an ideal generated by square-free monomials
$m_1,\ldots,m_q$. If $\ba=(a_1,\ldots,a_q)\in
\NN^q$, set
$$\bma={m_1}^{a_1}\cdots {m_q}^{a_q} \qand \pmea={\e_1}^{a_1}\cdots {\e_q}^{a_q}.$$
Observe
\begin{equation}
\label{e:pmea}
\pmea =\prod_{j\in [q]} \prod_{\substack{ A \subseteq [q]\\ j \in A}}  {x_A}^{a_j}= \prod_{\emptyset \neq A\subseteq [q]} (x_A)^{\sum_{j\in A}a_{j}}\,.
\end{equation}
The $r^{th}$ powers $I^r$ and $\Erq$ are generated by monomials of the
form $\bma$ and $\pmea$, respectively, with $\ba\in
\Nrq$. \cref{p:mingens} demonstrates that for $\Erq$, the elements
$\pmea$ with $\ba\in \Nrq$ form a minimal generating set. In fact, all
faces of $\Lrq$ are necessary in $\Lr(\E_q)$; i.e., none may be
removed when constructing $\Lr(\E_q)$.

\begin{proposition}\label{p:mingens}
  Let $r$ and $q$ be positive integers, and $\ba=(a_1,\ldots,a_q), \bb=(b_1,\ldots,b_q) \in \NN^q$. Then
  \begin{enumerate}
  \item[(i)] $\pmeb\mid \pmea$ $\iff$ $b_i \leq a_i$ for every $i\in [q]$.
  \item[(ii)] If $\ba, \bb\in \Nrq$,  then $\pmeb\mid \pmea$ $\iff$ $\bb=\ba$.
    \end{enumerate}
  In particular, if $\ba,\bb \in \Nrq$, then then $\Lrq=\Lr(\E_q)$.
\end{proposition} 
   
\begin{proof}
   To prove $(i)$, use \eqref{e:pmea} to justify the first equivalence below.
 $$\pmeb\mid \pmea \iff
  \sum_{j\in A}b_{j} \leq \sum_{j\in A}a_{j} \qforeach \emptyset \neq A \subseteq [q] \iff
  b_{j} \leq a_{j} \qforeach j \in [q].$$

  Now $(ii)$ follows from $(i)$ and the added assumption 
  $$b_1 + \cdots +b_q = a_1 + \cdots +a_q = r.$$

  The final claim follows directly from  \cref{d:LrI}. 
\end{proof}
   
In general, given the discussions above (see
also~\cite{M}), for every $r$ and $q$ we have
\begin{equation}\label{e:inclusions}
\LL^r(\E_q) = \Lrq \subseteq \Taylor(\Erq),
\end{equation}
where we assume that every vertex $\ba$ of $\LL^r_q$
  is labeled with a generator $\pmea$ of ${\E_q}^r$, and each face $\sigma$ is
  labeled with the lcm of the labels of its vertices, denoted $M_{\sigma}$.
  The following observation will be useful when working with the monomial label $M_{\sigma}$ of $\sigma\in \Taylor(\Erq)$. 
  \begin{remark}
  Let $\sigma=\{\pme^{\ba_1}, \dots, \pme^{\ba_t}\}\in \Taylor(\Erq)$ and set ${\ba_i}=(a_{i1}, a_{i2}, \dots, a_{iq})$ for $i\in [t]$. Using 
  \eqref{e:pmea}, we then have  
  \begin{equation}
  \label{e:Msigma}
  M_\sigma=\lcm\Big ( \prod_{\emptyset \neq A\subseteq [q]}(x_A)^{\sum_{j\in A}a_{ij}}: i\in [t]\Big )=\prod_{ \emptyset \neq A \subseteq [q]} (x_A)^{\underset{1\le i\le t}{\max}  \sum_{j\in A} a_{ij}}\,.
  \end{equation}
  Furthermore,  if $\bc=(c_1, \dots, c_q)\in \Nrq$, \eqref{e:pmea} and \eqref{e:Msigma} give
  \begin{equation}
  \label{e:divide-lcm}
\pme^{\bc}\mid M_\sigma\iff \sum_{j\in A}c_j\le \underset{1\le i\le t}{\max}  \sum_{j\in A} a_{ij}\qquad\text{for all}\quad \emptyset\ne A\subseteq [q].
  \end{equation}
\end{remark}  

We show in \cref{t:extremal-betti} that powers of extremal ideals
attain maximal Betti numbers among powers of all square-free monomial
ideals with the same number of generators. To prove this, we first
define a ring homomorphism $S_{\E}\to S$ and discuss its properties.
  
\begin{definition}[{\bf The ring homomorphism $\psi_I$}]\label{d:psi}
  Let $I$ be an ideal of the
  polynomial ring $S=\sfk[x_1,\ldots,x_n]$ minimally generated by
  square-free monomials $m_1,\ldots,m_q$. For each $k \in [n]$, set
  $$A_k= \{j \in [q] \st x_k \mid m_j\}\,.$$
  We have thus
  $$
j\in A_k\iff x_k\mid m_j\,.
  $$
  Define $\psi_I$ to be the
  ring homomorphism $$\psi_I \colon S_\E \to S \qwhere
  \psi_I(x_A)=\begin{cases}
  \displaystyle \prod_{\substack{k\in[n] \\ A=A_k}} x_k & \mbox{if } A=A_k \mbox{ for some
  } k \in [n],\\ 1 & \mbox{otherwise}.
  \end{cases}
$$
\end{definition}

Before proceeding directly with our work on extremal ideals, we illustrate how the map $\psi_I$ works in a sample case where there are four generators and seven variables. 

 \begin{example}\label{e:psi} 
   Let $I$ be the ideal of $\sfk[x_1, \ldots , x_7]$ generated by the
   square-free monomials
   $$m_1 = x_1x_2x_5x_7, \quad m_2 = x_2x_3x_7, \quad m_3 = x_3x_4x_6, \quad
   m_4 = x_1x_4.$$ Since $n = 7$ and $q = 4$, it follows that
   $$A_1=\{1,4 \}, \quad A_2 =\{1,2 \}, \quad A_3 =\{2,3 \}, \quad A_4 =\{3,4 \}, \quad A_5 =\{1 \}, \quad A_6 =\{3 \}, \quad A_7 =\{1,2 \}.$$

The function
$$\psi_I: \sfk\big{[}x_A\colon \emptyset\ne A\subseteq [q]\big{]} \to \sfk[x_1,
  \ldots, x_7]$$ maps 
$$\begin{array}{llll}
  \psi_I(x_{\{1,4\}}) = x_1,
& \psi_I(x_{\{1,2\}}) = x_2x_7, 
& \psi_I(x_{\{2,3\}}) = x_3,
& \psi_I(x_{\{3,4\}}) = x_4,\\  
&&&\\
  \psi_I(x_{\{ 1\}}) = x_5,
& \psi_I(x_{\{ 3\}}) = x_6,  & \mbox{and}  
& \psi_I(x_A) = 1 \,\, \mbox{otherwise}.     
\end{array}   
$$
Observe that 
$$
\psi_I(\e_1)=\psi_I \Big (\prod_{\substack{A\subseteq [q] \\ 1\in A}}x_A \Big)=\prod_{\substack{A\subseteq [q] \\ 1\in A}}\psi_I(x_A)=\psi_I(x_{\{1,4\}})\psi_I(x_{\{1,2\}})\psi_I(x_{\{1\}})=x_1(x_2x_7)x_5=m_1\,.
$$
\end{example}

Generalizing the properties of $\psi_I$ from the example, we arrive at the following lemma. 
 \begin{lemma}  Let $I$ be an ideal of the
   polynomial ring $S=\sfk[x_1,\ldots,x_n]$, minimally generated by
   square-free monomials $m_1,\ldots,m_q$. Then
   \begin{enumerate}
   \item[(i)] $\psi_I(\pmea)=\bma$ for each $\ba\in \Nrq$;
   \item[(ii)] $\psi_I(\Erq)S  =I^r$ for every $r>0$;
   \item[(iii)] $\psi_I$ preserves least
     common multiples, that is: 
$$
   \psi_I(\lcm(\pme^{\ba_1}, \dots, \pme^{\ba_t}))=\lcm(\bm^{\ba_1}, \dots, \bm^{\ba_t}) \qforall \ba_1,\ldots,\ba_t \in \Nrq, \quad t\ge 1.
$$     
   \end{enumerate}
 \end{lemma}
 
 \begin{proof}
   By
   \cref{d:extremal,d:psi}, for every $j\in[q]$
 $$
   \psi_I(\e_j)
   = \psi_I\Big (\prod_{\substack{A\subseteq [q] \\ j\in A}} x_A \Big)
   =\prod_{\substack{A\subseteq [q] \\ j\in A}} \psi_I(x_A)
   =\prod _{\substack{A\subseteq [q] \\ j\in A}} \prod_{\substack{k\in[n]\\A=A_k}} x_k
   =\prod_{\substack{k\in[n]\\j \in A_k}}x_k
   =\prod_{\substack{k\in[n]\\x_k \mid m_j}}x_k=m_j. 
 $$

It follows that $\psi_I(\pmea)=\bma$ for all $\ba\in \Nrq$, which
   establishes $(i)$ and also $(ii)$.  It remains to show $(iii)$. 
 
Set ${\ba_i}=(a_{i1}, a_{i2}, \dots, a_{iq})$, for $i\in [t]$. Using \eqref{e:Msigma} in the first equality below, we have: 

 \begin{align*}
\psi_I(\lcm(\pme^{\ba_1}, \dots, \pme^{\ba_t}))&=\psi_I\Big{(}\prod_{ \emptyset \neq A \subseteq [q]} (x_A)^{\underset{1\le i\le t}{\max}  \sum_{j\in A} a_{ij}}\Big{)}    \\
 &=\prod_{\emptyset \neq A \subseteq [q]} \Big (\prod_{\substack{k\in[n] \\ A=A_k}} x_k \Big )^{\underset{1\le i\le t}{\max} \sum_{j\in A} a_{ij}}\\
 &=\prod_{k\in [n], A_k\ne \emptyset}(x_k)^{\underset{1\le i\le t}{\max}\sum_{j\in A_k} a_{ij}}\\
 &=\lcm \Big (\prod_{k\in [n], A_k\ne \emptyset} (x_k)^{\sum_{j\in A_k}a_{1j}},\dots,  \prod_{k\in [n], A_k\ne \emptyset} (x_k)^{\sum_{j\in A_k}a_{tj}}\Big )\\
 &=\lcm \Big (\prod_{j\in [q]}\prod_{\substack{k\in[n]\\j \in A_k}} (x_k)^{a_{1j}}, \dots, \prod_{j\in [q]}\prod_{\substack{k\in[n]\\j \in A_k}}(x_k)^{a_{tj}} \Big)\\
 &=\lcm \Big(\prod_{j\in [q]}m_j^{a_{1j}}, \dots, \prod_{j\in [q]} m_j^{a_{tj}}\Big)\\
 &=\lcm(\bm^{\ba_1},\dots, \bm^{\ba_t}).
\end{align*} 
\end{proof}

\cref{t:extremal-betti} below demonstrates why the ideals from
\cref{d:psi} are called \textit{extremal}: they have the
greatest Betti numbers among all ideals minimally generated by $q$
square-free monomials.  The following lemma provides the technical
preliminaries necessary for the proof of \cref{t:extremal-betti}.

\begin{lemma} 
\label{l:Tors} Let $I$ be an ideal minimally generated by $q$ square-free monomials in a polynomial ring $\sfk[x_1, \ldots, x_n]$. With notation as in \cref{d:psi}, 
if $S$ is viewed as an $S_{\E}$-module via the ring homomorphism $\psi_I \colon S_{\E}\to S$, then 
$$S_{\E}/\Erq\otimes_{S_{\E}}S\cong S/I^r\quad\text{and}\quad \Tor_i^{S_{\E}}(S_{\E}/\Erq, S)=0\quad\text{for all $i>0$}.
$$ 
\end{lemma}

\begin{proof}
 First, note that 
$$
 \frac{S_{\E}}{\Erq}\otimes_{S_{\E}} S\cong \frac{S}{(\Erq)S}=\frac{S}{\psi_I(\Erq)S}=\frac{S}{I^r}\,.
 $$ To compute $\Tor_i^{S_{\E}}(S_{\E}/\Erq, S)$, use the Taylor
 complex $\Taylor(\Erq)$, which supports a free resolution $\mathbb F$
 of $S_{\E}/\Erq$; see \eqref{e:Taylor-diff} for a description of the
 differentials of $\mathbb F$.  Since the homomorphism $\psi_I$
 changes the labels $\pmea$ to the labels $\bma$ and preserves least
 common multiples, the chain complex $\mathbb F\otimes_{S_{\E}}S$ is
 isomorphic to a homogenization of the chain complex associated to the
 simplex with vertices corresponding to $\ba\in \Nrq$ and labeled with
 the monomials $\bma$. This is a version of the Taylor resolution of
 $S/I^r$ defined starting with a possibly non-minimal generating set
 of $I^r$. Such a non-minimal version is a free resolution of $S/I^r$
 as well, hence $\Tor_i^{S_{\E}}(S_{\E}/\Erq, S)=0$ when $i>0$.  (See
 \cref{r:larger-Taylor}.)
\end{proof}

  \begin{theorem}[{\bf Powers of extremal ideals have maximal Betti numbers}]
\label{t:extremal-betti} Given positive integers $r$ and $q$, 
$$\beta_i^S(I^r)\le \beta_i^{S_{\E}}(\Erq)$$ for any $i\ge 0$ and any
ideal $I$ minimally generated by $q$ square-free monomials in a polynomial ring
$S$.\end{theorem}

\begin{proof}
 Let $\mathbb F$ be a minimal free resolution of $S_{\E}/\Erq$ over $S_{\E}$. Then  \cref{l:Tors} shows that $\mathbb F\otimes_{S_{\E}}S$ is a free resolution of $S/I^r$ over $S$. Consequently,
$$
\beta_i^{S_{\E}}(\Erq)=\text{rank}_{S_{\E}}(\mathbb F_{i+1})=\text{rank}_{S}(\mathbb F_{i+1}\otimes_{S_{\E}}S)\ge \beta_i^S(I^r).
$$
\end{proof} 

 In view of \cref{t:extremal-betti}, the homogenized chain complex of any (cell) complex that supports a minimal free resolution of $\Erq$ can be thought of as an upper bound for the minimal free resolution of the $r^{th}$ power of any ideal minimally generated by $q$ square-free monomials. \cref{p:minimal} establishes when our simplicial complex $\Lrq$ supports a minimal free resolution of $\Erq$. 

\begin{proposition}[{\bf When $\Lrq$ supports a  minimal free resolution of $\Erq$}]\label{p:minimal}
Let $r$ and $q$ be positive integers. The following statements are equivalent: 
\begin{enumerate}
\item The simplicial complex $\Lrq$
supports a minimal free resolution of the ideal $\Erq$.
\item $\pd_{S_\E}(\Erq)=\dim \Lrq$. 
\item One of the following conditions holds:
\begin{itemize}
\item  $q=1$ and $r\ge 1$;
\item  $q=2$ and $1\le r\le 4$;
\item  $q\ge 3$ and $1\le r\le 2$.
\end{itemize}
\end{enumerate}
\end{proposition}

\begin{proof}

(1)$\implies$(2): This implication is clear.

(2)$\implies$(3): Assume that $\pd_{S_\E}(\Erq)=\dim \Lrq$ and the values of $r$ and $q$ do not satisfy the conditions in (3). In particular, we must have $q\ge 2$ and $r\ge 5$, or $q\ge 3$ and  $r=3,4$. By \cref{r:f>q} and \cref{e:f}, for all $i,j \in [q]$,  
\begin{equation}\label{e:star}
|F_i^r| = f > q =|G_j^r|
\end{equation}
and so $F_1^r, \ldots, F_q^r$ are the facets of $\Lrq$ of highest dimension, with the caveat that $F_1^r=\cdots =F_q^r$ when $r=3=q$, and thus $\dim \Lrq=f-1$. Since  we assumed $\pd_{S_\E}(\Erq)=\dim \Lrq$, we have $\pd_{S_\E}(\Erq)=f-1$ as well.  Let $\mathbb F$ denote the free resolution supported on $\Lrq$, and let $\partial$ denote its differential, which is described in \eqref{e:Taylor-diff}.  Since $\pd_{S_\E}(\Erq)=f-1$ and a minimal free resolution of $\Erq$ is isomorphic to a direct summand of $\mathbb F$, there must exist a basis element $e\in \mathbb F_{f-1}\smallsetminus (S_\E)_{\geqslant 1}\mathbb F_{f-1}$ such that 
\begin{equation}
\label{e:min}
\partial(e)\in (S_\E)_{\geqslant 1}\mathbb F_{f-2}\,.
\end{equation}
As in \eqref{e:Taylor-diff}, let $e_\sigma$ denote the basis element in $\mathbb F$ corresponding to $\sigma\in \Lrq$.  We write
$$
e=\sum_{F\in \Lrq, |F|=f} \alpha_Fe_{F}
$$
with $\alpha_F\in S_\E$. The assumption  $e\notin(S_\E)_{\geqslant 1}\mathbb F_{f-1}$ implies that $\alpha_F$ is a unit for some $F\in \Lrq$ with $|F|=f$. By \cref{e:star}, we see that  $F=F_i^r$ for some $i\in [q]$. Without loss of generality, assume $i=1$. 

Recall that $M_\sigma$ denotes the $\lcm$ of the monomial labels of the vertices in $\sigma\in \Lrq$.
\begin{claim}
 There exists $\bc\in F_1^r\smallsetminus G_1^r$ such that $M_{F_1^r}=M_{F_1^r\smallsetminus \{\bc\}}$.
\end{claim}

\noindent {\it Proof of Claim.} Assume first  $r>4$, and $q
\geq 2$, so that $s=\left\lceil{\frac{r}{2}}\right\rceil \geq 3$. Let
$\ba,\bb, \bc \in \Nrq$ be such that
$$\pmea={\e_1}^{r-s}{\e_2}^{s}, \quad
  \pmeb={\e_1}^{r-s+2}{\e_2}^{s-2} \qand
  \pme^{\bc}={\e_1}^{r-s+1}{\e_2}^{s-1},
$$ 
so that $\ba, \bb, \bc$ are distinct vertices of $\Lrq$.
Note that
$
\pme^{\bc} \mid \lcm(\pmea,\pmeb)
$. 
Indeed, after removing the common factors, this divisibility is equivalent to 
$$\e_1\e_2\mid \lcm(\e_1^2, \e_2^2)\,,$$
and can be verified using \eqref{e:divide-lcm}. 

Now let $r=3$ or $r=4$, and $q\ge 3$. Then $s=2$, and if one sets
$$\pmea={\e_1}^{r-2}{\e_2}^{2}, \quad 
\pmeb={\e_1}^{r-1}{\e_3} \qand
\pme^{\bc}={\e_1}^{r-2}{\e_2}\e_3,
$$ then we see that $\pme^{\bc} \mid \lcm(\pmea,\pmeb)$. Indeed, after removing the common factors, this divisibility is equivalent to 
$$
\e_2\e_3\mid \lcm(\e_2^2, \e_1\e_3)\,,
$$
and can be verified using \eqref{e:divide-lcm}. 

In both cases, we have $\ba,\bb, \bc\in F_1^r$ and $\bc\notin
G_1^r$. The divisibility $\pme^{\bc} \mid \lcm(\pmea,\pmeb)$
establishes the conclusion of the Claim.

We finish the proof of (2)$\implies$(3) as follows.  We have 
\begin{equation}
\label{e:partial}
\partial(e)=\sum_{F\in \Lrq, |F|=f} \alpha_F\partial(e_{F})=\sum_{F\in \Lrq, |F|=f} \alpha_F\left(\sum_{\bc'\in F}\pm \frac{M_F}{M_{F\smallsetminus \{\bc'\}}}e_{F\smallsetminus \{\bc'\}}\right)\,.
\end{equation}
Let $\bc $ be as in the Claim and let $F\in \Lrq$ with $|F|=f$. By \cref{e:star}, we have $F=F_j^r$ for some $j$.  If $\bc'\in F$, observe
\begin{equation}
\label{e=}
F_1^r\smallsetminus \{\bc\}=F\smallsetminus\{\bc'\}\iff \text{ $F_1^r=F$ and $\bc=\bc'$.}
\end{equation}
To prove this statement, we refer to \cref{sec:Lrq} for basic properties of the sets $F_i^r$, $G_i^r$ and $B^r$. Indeed, \eqref{e=} is clear when  $r=3$, since  $F_1^r=B^r=F_j^r$ for all $j\in [q]$ in 
this case.  Assume now $r>3$, and recall that $q\ge 2$.  Assuming $F_1^r\smallsetminus \{\bc\}=F_j^r\smallsetminus\{\bc'\}$ and $j\ne 1$, we use  $F_1^r\cap F_j^r=B^r$, and conclude
$
F_1^r=B^r\cup \{\bc\}$. 
Since $B^r\cap G_1^r=\emptyset$ and $\bc\notin G_1^r$, this contradicts the fact that $F_1^r\cap G_1^r\ne\emptyset$ when $r>3$ and $q\ge 2$. Thus, \eqref{e=} must hold. 

In view of the Claim, \eqref{e=} and \eqref{e:partial}, we see that the coefficient of $e_{F_1^r\smallsetminus \{\bc\}}$ in $\partial(e)$ is equal to  $\pm \alpha_{F_1^r}$ hence it is a unit. This contradicts \eqref{e:min}.

  (3)$\implies (1)$: \cref{t:main} shows that $\Lrq$ supports a
  free resolution of $\Erq$ for all $r,q \ge 1$.  To show minimality, according to
  \cref{t:BPS}, it suffices to show that $M_{\sigma}\ne M_{\sigma'}$
  for any faces $\sigma, \sigma'\in \Lrq$ with $\sigma\ne \sigma'$,
  or, in other words, that each monomial label appears only once.
  
  If $q=1$, then $I^r=({m_1}^r)$ for all $r$, and all complexes in
  \eqref{e:inclusions} are one point, so each supports a minimal
  resolution by default.
    
 If $q=2$, then $\E_2=(x_{1}x_{12},x_{2}x_{12})$, and  $\LL^2(\E_2)$,
 $\LL^3(\E_2)$, and $\LL^4(\E_2)$ are shown below. Observe that each monomial
 label appears once in each complex, hence $\LL^2(\E_2)$,
 $\LL^3(\E_2)$, and $\LL^4(\E_2)$ support a minimal free resolution of
 ${\E_2}^2$, ${\E_2}^3$ and ${\E_2}^4$, respectively.

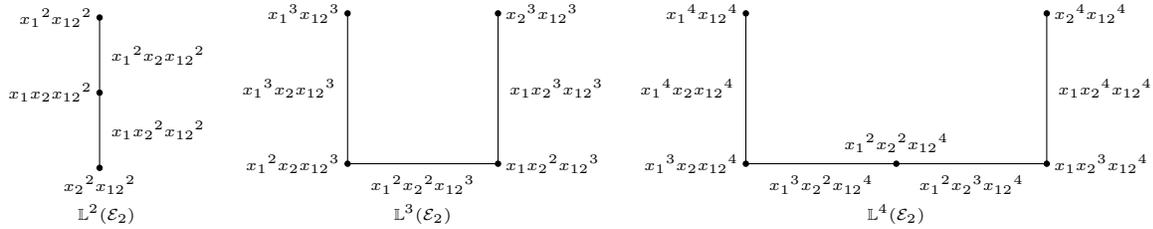
\begin{center}
\tiny{
\begin{tabular}{ccc}
\begin{tikzpicture}
\tikzstyle{point}=[inner sep=0pt]
\node (a)[point,label=left:${x_{1}}^2{x_{12}}^2$] at (0,2) {};
\node (b)[point,label=left:${x_{1}}{x_{2}}{x_{12}}^2$] at (0,1) {};
\node (c)[point,label=below:${x_{2}}^2{x_{12}}^2$] at (0,0) {};
\node (g)[label=right: ${x_{1}}^2{x_{2}}{x_{12}}^2$] at (0,1.5) {};
\node (e)[label=right:  ${x_{1}}{x_{2}}^2{x_{12}}^2$] at (0,.5) {};
\draw (a.center) -- (b.center);
\draw (b.center) -- (c.center);
\filldraw [black] (a.center) circle (1pt);
\filldraw [black] (b.center) circle (1pt);
\filldraw [black] (c.center) circle (1pt);
\end{tikzpicture}
&
\begin{tikzpicture}
\tikzstyle{point}=[inner sep=0pt]
\node (a)[point,label=left:${x_{1}}^3{x_{12}}^3$] at (-1,2) {};
\node (b)[point,label=left:${x_{1}}^2{x_{2}}{x_{12}}^3$] at (-1,0) {};
\node (c)[point,label=right:${x_{1}}{x_{2}}^2{x_{12}}^3$] at (1,0) {};
\node (d)[point,label=right:${x_{2}}^3{x_{12}}^3$] at (1,2) {};
\node (g)[label=right: ${x_{1}}{x_{2}}^3{x_{12}}^3$] at (1,1) {};
\node (e)[label=left:  ${x_{1}}^3{x_{2}}{x_{12}}^3$] at (-1,1) {};
\node (f)[label=below:  ${x_{1}}^2{x_{2}}^2{x_{12}}^3$] at (0,0) {};
\draw (a.center) -- (b.center);
\draw (b.center) -- (c.center);
\draw (c.center) -- (d.center);
\filldraw [black] (a.center) circle (1pt);
\filldraw [black] (b.center) circle (1pt);
\filldraw [black] (c.center) circle (1pt);
\filldraw [black] (d.center) circle (1pt);
\end{tikzpicture}
&
\begin{tikzpicture}
\tikzstyle{point}=[inner sep=0pt]
\node (a)[point,label=left:${x_{1}}^4{x_{12}}^4$] at (-2,2) {};
\node (b)[point,label=left:${x_{1}}^3{x_{2}}{x_{12}}^4$] at (-2,0) {};
\node (c)[point,label=right:${x_{1}}{x_{2}}^3{x_{12}}^4$] at (2,0) {};
\node (d)[point,label=right:${x_{2}}^4{x_{12}}^4$] at (2,2) {};
\node (g)[label=right: ${x_{1}}{x_{2}}^4{x_{12}}^4$] at (2,1) {};
\node (e)[label=left:  ${x_{1}}^4{x_{2}}{x_{12}}^4$] at (-2,1) {};
\node (f)[label=above:  ${x_{1}}^2{x_{2}}^2{x_{12}}^4$] at (0,0) {};
\node (g)[label=below:  ${x_{1}}^3{x_{2}}^2{x_{12}}^4$] at (-1,0) {};
\node (h)[label=below:  ${x_{1}}^2{x_{2}}^3{x_{12}}^4$] at (1,0) {};
\draw (a.center) -- (b.center);
\draw (b.center) -- (f.center);
\draw (f.center) -- (c.center);
\draw (c.center) -- (d.center);
\filldraw [black] (a.center) circle (1pt);
\filldraw [black] (b.center) circle (1pt);
\filldraw [black] (c.center) circle (1pt);
\filldraw [black] (d.center) circle (1pt);
\filldraw [black] (f.center) circle (1pt);
\end{tikzpicture}\\

$\LL^2(\E_2)$&$\LL^3(\E_2)$& $\LL^4(\E_2)$
\end{tabular} }
\end{center} 

Now assume $q\geq 3$ and $1 \leq r \leq 2$. We will show that, for
every $\bc \in \Nrq$ and $\sigma\in \Lrq$,
\begin{equation}\label{e:sigma} \pmec \mid M_\sigma \iff \bc
    \in \sigma.
  \end{equation}

 When $r=1$ observe that for every $i \in [q]$,
  \begin{equation}
\label{e:e}
\e_i\nmid \lcm(\e_{k_1}, \dots, \e_{k_p}) \qforall i\in
  [q], \quad k_1, \dots, k_p\in [q]\smallsetminus \{i\}.
  \end{equation}
  This can be seen by noting that the right-hand side of
  \eqref{e:divide-lcm} becomes $1\le 0$ for $A=\{i\}$, $\pmec=\e_i$
  and $\sigma=\{\e_{k_1}, \dots, \e_{k_p}\}$. Therefore
  $\LL^1_q=\LL^1(\E_q)=\Taylor(\E_q),$ and \eqref{e:sigma} follows
  immediately.

  Now let $r=2$, $q\geq 3$, $\sigma \in \LL^2_q$ and $\bc \in \N_q^2$, with 
  $\pmec\mid M_\sigma$. Since $r=2$, $\pmec=\e_i\e_j$ for some $i,j \in [q]$.
  Pick $A=\{i,j\}$ in the right-hand side of \eqref{e:divide-lcm}.
  
  \begin{itemize}
 \item If $i=j$, then $\pmec={\e_i}^2$ and $A=\{i\}$. By
   \eqref{e:divide-lcm} we must have $\bc\in \sigma$.

 \item If $i \neq j$, then by \eqref{e:divide-lcm} there exists $\bb
   \in \sigma$ with
 $$\pmeb=\e_i\e_j, \ {\e_i}^2 \mbox{ or } {\e_j}^2.$$ If
   $\pmeb=\e_i\e_j$ then $\bc=\bb\in\sigma$, as desired.  Suppose $\pmec \notin \sigma$, so without loss of generality 
   $\pmeb={\e_i}^2$.  As $\sigma\in \LL^2_q$, we must
   have $\sigma\subseteq G_i^2$. So there exist $k_1, \dots, k_p$
   in $[q]\smallsetminus \{i,j\}$ such that
$$M_\sigma=\lcm({\e_i}^2, \e_i\e_{k_1}, \dots,
  \e_i\e_{k_p})=\e_i\lcm(\e_i,\e_{k_1}, \dots, \e_{k_p})\,.
$$
Since $\e_i\e_j\mid M_\sigma$, it follows that $\e_j\mid \lcm(\e_i,\e_{k_1}, \dots, \e_{k_p})$, which contradicts \eqref{e:e}. Thus $\bc \in \sigma$.

  \end{itemize}
  
 We have now proved the statement in~\eqref{e:sigma}, and by
 \cref{t:BPS} we conclude that $\LL^2_q$ supports a minimal free
 resolution of ${\E_q}^2$ for every $q\geq 3$.
 
\end{proof}

A direct consequence of our work in  \cref{p:minimal} is the following statement.

\begin{proposition}[{\bf When $\Lrq$ supports a  minimal resolution of some $I^r$}]\label{p:last}
  If  $r$ and $q$ are positive integers, then $\Lrq $ supports a minimal free resolution of
   $I^r$ for some ideal $I$ minimally generated by $q$ square-free monomials if and only if one of the following holds:
\begin{itemize}
\item  $q=1$ and $r\ge 1$;
\item  $q=2$ and $1\le r\le 4$;
\item  $q\ge 3$ and $1\le r\le 2$.
\end{itemize}
\end{proposition}

\begin{proof} 
For any square-free monomial ideal
$I$ with $q$ minimal generators the Betti numbers of $I^r$ are
  bounded above by the  Betti numbers of
  $\Erq$ by \cref{t:extremal-betti}. So the question is reduced to when
  $\Lrq = \LL ^r (\E_q)$ supports a minimal free resolution of $\Erq$. The
  rest follows from \cref{p:minimal}.
\end{proof}

When $r=2$, the
fact that these bounds are sharp had been previously announced in
\cite{L2}. The search for sharp(er) bounds when $r>2$ is
continued in \cite{MFO}.
%%%%%%%%%%%%%%%%%%%%%%%%%%%%%%%%%%%%%%%%%%%%%%%%%%%%%%%%%%%%
%%%%%%%%%%%%%%%%%%%%%%%%%%%%%%%%%%%%%%%%%%%%%%%%%%%%%%%%%%%%

\subsubsection*{Acknowledgments} 

The research leading to this paper was initiated during the week long
workshop ``Women in Commutative Algebra'' (19w5104) which took place
at the Banff International Research Station (BIRS). The authors would
like to thank the organizers and acknowledge the
hospitality of BIRS and the additional support provided by the
National Science Foundation, DMS-1934391.

The authors are extremely grateful to the anonymous referee for their careful
reading of this paper, excellent suggestions, as well as detecting and
correcting a mistake in the proof of
\cref{p:irredundant-gens-Ir}. 

For this work Liana \c Sega was supported in part by a grant from the
Simons Foundation (\#354594), and Susan Cooper and Sara Faridi were
supported by the Natural Sciences and Engineering Research Council of
Canada (NSERC).

The computations for this project were done using
the computer algebra program Macaulay2~\cite{M2}.

For the last author, this material is based upon work
  supported by and while serving at the National Science
  Foundation. Any opinion, findings, and conclusions or
  recommendations expressed in this material are those of the authors
  and do not necessarily reflect the views of the National Science
  Foundation.

%%%%%%%%%%%%%%%%%%%%%%%%%%%%%%%%%%%%%%%%%%%%%%%%%%%%%%%%%%%%
%%%%%%%%%%%%%%%%%%%%%%%%%%%%%%%%%%%%%%%%%%%%%%%%%%%%%%%%%%%%


\begin{thebibliography}{2020}

\bibitem{BPS} D.~Bayer, I.~Peeva, B.~Sturmfels, {\it Monomial
  resolutions}, Math. Res. Lett. 5, no.~1-2 (1998) 31--46.

\bibitem{BS} D.~Bayer, B.~Sturmfels, {\it Cellular resolutions of
  monomial modules}, J. Reine Angew. Math. 503 (1998) 123--140.
  
\bibitem{morse} S.~M.~Cooper, S.~El Khoury, S.~Faridi, S~Mayes-Tang,
  S.~Morey, L.~M.~\c{S}ega, S.~Spiroff, {\it Morse resolutions of
    powers of square-free monomial ideals of projective dimension one}, 
    J. Algebraic Combin. 55 (2022), no. 4, 1085--1122.

\bibitem{L2} S.~M.~Cooper, S.~El Khoury, S.~Faridi, S~Mayes-Tang,
  S.~Morey, L.~M.~\c{S}ega, S.~Spiroff, {\it Simplicial resolutions
    for the second power of square-free monomial ideals}, Women in
  Commutative Algebra - Proceedings of the 2019 WICA Workshop (2021), 193--205.
  
\bibitem{MFO} S.~El Khoury, S.~Faridi,
   L.~M.~\c{S}ega, S.~Spiroff, {\it Resolutions of powers of extremal
    ideals}, to appear in J. of Pure \&  Appl. Algebra. 
  
\bibitem{E} D.~Eisenbud, {\it The Geometry of Syzygies}, Graduate Texts in Mathematics 229. Springer Science+Business Media, Inc., New York, 2005.
  
\bibitem{EMO} J.~Eagon, E.~Miller, E.~Ordog, {\it Minimal resolutions of monomial ideals}, ArXiv:1906.08837.

\bibitem{F02} S.~Faridi, {\it The facet ideal of a simplicial
  complex}, Manuscripta Math. 109, no.~2 (2002) 159-174.

\bibitem{F14} S.~Faridi, {\it Monomial resolutions supported by
  simplicial trees}, J. Commut. Algebra 6, no.~ 3 (2014).
  
\bibitem{FH} S.~Faridi, B.~Hersey, {\it Resolutions of monomial ideals
  of projective dimension 1}, Comm. Algebra 45, no.~12 (2017)
  5453--5464.

\bibitem{M2} D.~R.~Grayson, M.~E.~Stillman, {\it Macaulay2, a software
  system for research in algebraic geometry}, Available at {\it
  http://www.math.uiuc.edu/Macaulay2/}.

\bibitem{L} G.~Lyubeznik, {\it A new explicit finite free resolution
  of ideals generate by monomials in an $R$-sequence}, J. Pure
  Appl.~Algebra 51 (1998) 193-195.

\bibitem{M} J. Mermin, {\it Three simplicial resolutions}, Progress in
  commutative algebra 1, 127-141, de Gruyter, Berlin, (2012).

\bibitem{MS} E.~Miller, B.~Sturmfels, {\it Combinatorial Commutative Algebra}, Graduate Texts in Mathematics 227. Springer Science+Business Media, Inc., New York, 2005.
  
\bibitem{NR} U.~Nagel, V.~Reiner, {\it Betti numbers of monomial
  ideals and shifted skew shapes}, Electron. J. Combin. 16 (2009),
  no. 2, Special volume in honor of Anders Bjorner, Research Paper 3.

 \bibitem{OW} P.~Orlik, V.~Welker, {\it Algebraic combinatorics.
  Lectures from the Summer School held in Nordfjordeid, June 2003},
  Universitext. Springer, Berlin, (2007).

 
\bibitem{P} I.~Peeva, {\it Graded syzygies}, Algebra and
  Applications, 14. Springer-Verlag London, Ltd., London, 2011.
  
\bibitem{PV} I.~Peeva, M.~Velasco, {\it Frames and degenerations of
  monomial ideals}, Trans. Amer. Math. Soc. 363 (2011), no. 4,
  2029--2046.

\bibitem{T} D.~Taylor, {\it Ideals generated by monomials in an
  $R$-sequence,} Thesis, University of Chicago (1966).

\bibitem{Tch} A.~Tchnerev, {\it Dynamical systems on chain complexes
  and canonical minimal resolutions}, ArXiv:1909.08577.

\bibitem{V} M.~Velasco, {\it Minimal free resolutions that are not supported by a CW-complex}, J. Algebra 319, no.~1 (2008) 102-114.
  
\bibitem{villa} R. H. Villarreal, {\it Rees algebras of edge ideals}, Comm. Algebra  23  (1995),  3513--3524.
  
\bibitem{Z} X.~Zheng, {\it Resolutions of facet ideals}, Comm. Algebra
  32 no.~6 (2004) 2301-2324.


\end{thebibliography}
\end{document}